\documentclass[reqno]{amsart}

\usepackage{mathrsfs}
\usepackage{amsmath}
\usepackage{amssymb}
\usepackage{bm}
\usepackage{cite}
\usepackage{latexsym}
\usepackage{graphicx}

\usepackage{color}
\usepackage{comment}

\theoremstyle{plain}
\newtheorem{thm}{Theorem}[section]
\newtheorem{lem}[thm]{Lemma}

\newtheorem{rmk}[thm]{Remark}

\def\F{\mathscr{F}}
\def\L{\mathscr{L}}
\def\U{\mathscr{U}}

\def\d{\mathrm{d}}

\def\Eset{\mathbb{E}}
\def\Nset{\mathbb{N}}
\def\Pset{\mathbb{P}}
\def\Rset{\mathbb{R}}
\def\Sset{\mathbb{S}}
\def\Zset{\mathbb{Z}}

\def\Span{\mathrm{span}}

\def\epsilon{\varepsilon}

\DeclareMathOperator{\esssup}{\mathrm{ess\,sup}}

\makeatletter
\@addtoreset{equation}{section}
\makeatother
\def\theequation{\arabic{section}.\arabic{equation}}

\begin{document}


\title[Continuum limits of coupled oscillator networks]%
{Continuum limits of coupled oscillator networks depending on multiple sparse graphs}
\thanks{This work was partially supported by JSPS KAKENHI Grant Number JP17H02859.}

\author[Ryosuke Ihara]{Ryosuke Ihara${}^\dag$}\thanks{%
${}^\dag$Present address:
 Hanshin Branch Office, Kubota Corporation, 1-1-1 Hama, Amagasaki 661-8567, Japan}

\author{Kazuyuki Yagasaki}

\address{Department of Applied Mathematics and Physics, Graduate School of Informatics,
Kyoto University, Yoshida-Honmachi, Sakyo-ku, Kyoto 606-8501, JAPAN}
\email{ihara.ryosuke.o43@kyoto-u.jp}
\email{yagasaki@amp.i.kyoto-u.ac.jp}

\date{\today}

\begin{abstract}
The continuum limit provides a useful tool for analyzing coupled oscillator networks.
Recently, Medvedev (Comm. Math. Sci., 17 (2019),
no.~4, pp.~883--898) gave a mathematical foundation
for such an approach when the networks are defined on single graphs
which may be dense or sparse, directed or undirected, and deterministic or random.
In this paper, we consider coupled oscillator networks depending on multiple graphs,
and extend his results to show that the continuum limit is also valid in this situation.
Specifically, we prove that the initial value problem (IVP)
 of the corresponding continuum limit has a unique solution under general conditions
 and that the solution becomes the limit of those to the IVP of the networks
 in some adequate meaning.
Moreover, we show that if solutions to the networks are stable or asymptotically stable
 when the node number is sufficiently large,
 then so are the corresponding solutions to the continuum limit,
 and that if solutions to the continuum limit are asymptotically stable,
 then so are the corresponding solutions to the networks in some weak meaning
 as the node number tends to infinity. 
These results can also be applied to coupled oscillator networks with multiple frequencies
 by regarding the frequencies as a weight matrix of another graph.
We illustrate the theory for three variants of the Kuramoto model
 along with numerical simulations.
\end{abstract}
\maketitle


\section{Introduction}

Coupled oscillator networks on graphs provide many mathematical models
such as neural networks \cite{LC01,MZ12}, Josephson junctions \cite{PVWO93},
power networks \cite{DB12} and consensus protocols \cite{M12}.
Understanding their dynamics and developing efficient control methods for them
are of importance in applied sciences and engineering,
and they are challenging problems especially because of the diversity of underlying graphs
such as small-world and scale-free properties.
These systems are difficult to treat and analyze
since they are generally of very high dimension and often nonlocally coupled
\cite{GHM12,KB02,OA08,SK04,TK03,WS98,WSG06}.
In this situation, continuum limits provide useful tools
for analyzing nonlocally coupled oscillator networks \cite{GHM12,KB02,OA08,WSG06}.
In the continuum limits,
solutions to highly dimensional systems of differential equations
are approximated by those to single integro-differential equations.
They were successfully used to investigate many interesting phenomena
 such as chimera states \cite{AS06,KB02}, multistability \cite{GHM12,WSG06},
 synchronization \cite{RLZ13,R08}, and coherence-incoherence transition \cite{ORHMS12,OWYMS12}.
Recently, Medvedev \cite{M14a,M14b} gave a mathematical foundation
 for such an approach to networks defined on single graphs
 which may be deterministic or random.
Moreover, he extended these results in \cite{M19}
 to a more general class of graphs containing dense, sparse, directed and undirected ones
 after giving some partial results in \cite{KM17} with his coworker.

A different approach for approximation of coupled oscillator networks by continuum models
 was utilized based on mathematical foundations in \cite{CM19a,CM19b,CMM18}.
Integro-partial differential equations called the \emph{Vlasov equations} were analyzed,
 and the previous results of \cite{C15,CN11} for complete graphs, i.e., all-to-all coupling,
 were extended to general deterministic and random graphs there.
Compared with those results, where probability density functions are treated,
 an advantage of Medvedev's result \cite{M19}
 is to guarantee the almost sure convergence of solutions of the coupled oscillators
 to those of the deterministic continuum limits
 even though the oscillators are defined on random networks.
Thus, his result can give more precise description
 on the dynamics of the coupled oscillators networks if it works although not always.
We now state some details of his result.

Let $G_n=\langle V(G_n),E(G_n),W(G_n)\rangle$, $n\in\Nset$,
be a sequence of weighted graphs,
where $V(G_n)=[n]:=\{1,2,\ldots,n\}$ and $E(G_n)$ are the sets of nodes and edges, respectively,
and $W(G_n)$ is an $n\times n$ weight matrix given by
\begin{equation*}
(W(G_{n}))_{ij}=
\begin{cases}
w_{ij}^n & \mbox{if $(i,j)\in E(G_{n})$};\\
0 &\rm{otherwise.}
\end{cases}
\end{equation*}
The edge set is given by
\[
E(G_n)=\{(i,j)\in[n]^2\mid (W(G_{n}))_{ij}\neq 0\},
\]
where each edge is represented by an ordered pair of nodes $(i,j)$,
which is also denoted by $j\to i$, and a loop is allowed.
If $W(G_n)$ is symmetric,
then $G_n$ represents an undirected weighted graph
and each edge is also denoted by $i\sim j$ instead of $j\to i$.
When $G_n$ is a simple graph,
 $W(G_n)$ is a matrix whose elements are $\{0,1\}$-valued.
We call the graph $G_n$ \emph{\color{black}random}
 if $W(G_n)$ is a random matrix, and \emph{deterministic} otherwise.
Random graphs treated in this paper are simple 
 as in the previous work \cite{M14a,M14b,M19}.
See Section~2.
We say that $G_n$ is a \emph{dense} graph
if $|E(G_n)|/|V(G_n)|^2\rightarrow c$ as $n \rightarrow \infty$ for some constant $c>0$.
If $|E(G_n)|/|V(G_n)|^2\rightarrow 0$ as $n \rightarrow \infty$,
then we call it a \emph{sparse} graph.

We now consider a coupled oscillator network defined on the graph $G_n$,
\begin{equation}
\frac{d}{dt} u_i^n(t) = f ( u_i^n(t), t ) + \frac{1}{n\alpha_n}\sum^{n}_{j=1}
w_{ij}^nD( u_j^n(t) - u_i^n(t) ),\quad
i \in [n],
\label{eqn:ds1}
\end{equation}
where $u_i^n \colon \Rset \rightarrow\Rset$ 
stands for the phase of oscillator at the node $i \in [n]$ 
and $\alpha_n>0$ is a scaling factor which is one if $G_n$ is dense
 and less than one with $\alpha_{n}\searrow 0$ and $n \alpha_{n} \to\infty$
 as $n \rightarrow \infty$, if $G_n$ is sparse.
Moreover, $f(u,t)$ is Lipschitz continuous in $u$ and continuous in $t$,
and $D(u)$ is bounded and Lipschitz continuous.
For instance, when $f(u,t)=\omega$ ($=$const.), $D(u)=\sin u$ and $w_{ij}^n=1$ ($\forall i,j \in[n]$),
Eq.~\eqref{eqn:ds1} becomes a special case of the Kuramoto model \cite{K84},
\begin{equation}
\frac{d}{dt}u_i^n=\omega+\frac{1}{n}\sum^{n}_{j=1}\sin( u_j^n(t)-u_i^n(t) ),
\label{eqn:Km}
\end{equation}
where $u_i^n:\Rset\to\Sset^1=\Rset/2\pi\Zset$, $i\in[n]$.
Here $f(u,t)$ and $D(u)$ are assumed to be $2\pi$-periodic 
 and the ranges of $u_i^n$, $i\in[n]$, are changed from $\Rset$ to $\Sset^1$.
Note that Eq.~\eqref{eqn:Km} cannot be written in the form of \eqref{eqn:ds1}
 if the natural frequency $\omega$ depends on the node $i$.
The continuum limit of \eqref{eqn:ds1} is given by
\begin{equation}
\frac{\partial}{\partial t}u(t,x)=
f(u(t,x),t) + \int _{I} W(x,y) D(u(t, y)-u(t,x)) dy,
\label{eqn:cs1}
\end{equation}
where $u:\Rset\times I\rightarrow\Rset$, $I:=[0,1]$,
and $W(x,y)$ is an $L^2$ function on $I^2$
and represents some kind of limit of the weight matrix $W(G_n)$
(see Sections~2 for the details).
Medvedev and his coworker \cite{KM17,M19} proved that under certain general conditions,
 there exists a unique solution to the initial value problem (IVP) of \eqref{eqn:cs1}
 and it approximates the solution to the IVP of \eqref{eqn:ds1}:
The former is the limit of the latter in some adequate meaning as $n\to\infty$.

On the other hand, the control problem of these nonlinear oscillator networks
is also important in applications \cite{GW15,SA15,SA16}.
One approach of controlling the coupled oscillator network \eqref{eqn:ds1}
 is to apply control force
\begin{equation}
\tilde{f}(\bar{U}(t)-u_i^n(t))
\label{eqn:tf}
\end{equation}
to each oscillator for exhibiting a desired motion,
 e.g., $u_i^n(t)=\bar{U}(t)$ or $|u_i^n(t)-\bar{U}(t)|\gg 1$ for every $i\in[n]$,
 where $\tilde{f}(u)$ is a different (desirably Lipschitz continuous) function from $f(u,t)$.
A similar general example was considered in \cite{SA15,SA16} (cf. Section~4).
Another approach is to apply the control force
\[
\frac{1}{n} \sum^{n}_{j=1} \tilde{w}_{ij}^n \tilde{D} (u_j^n(t) - u_i^n(t))
\]
to each oscillator for improving the performance of the network as desired,
 where $\tilde{w}_{ij}^n$ is the $(i,j)$-element of the weight matrix $\tilde{W}(\tilde{G}_n)$
for another weighted graph $\tilde{G}_n$, which is generally different from $G_n$,
and $\tilde{D}(u)$ is a different (desirably Lipschitz continuous) function from $D(u)$.
Then Eq.~\eqref{eqn:ds1} is modified to
\begin{align}
\frac{d}{dt} u_i^n(t)
=& f ( u_i^n(t), t ) +
\frac{1}{n\alpha_n} \sum^{n}_{j=1} w_{ij}^n D( u_j^n(t) - u_i^n(t) ) \notag\\
&+ \frac{1}{n\alpha_n'} \sum^{n}_{j=1} \tilde{w}_{ij}^n \tilde{D} ( u_j^n(t) - u_i^n(t) ) ,\quad
i \in [n].
\label{eqn:ds2}
\end{align}
Such a control approach depending on an additional network
 was also used in \cite{GC20,GC21}.
Compared with \eqref{eqn:ds1},
 the characteristics of \eqref{eqn:ds2} become more diverse, e.g.,
 attractive and repulsive nodes can be contained for each node.
A variant of the Kuramoto of this type is called a \emph{two-layer multiplex Kuramoto model}
 and was numerically studied in \cite{ST15}.
We expect that the continuum limit approach is also valid even for \eqref{eqn:ds2}
although it has not been proved.

In this paper, we consider more general nonlinear oscillator networks
 depending on multiple graphs $\{G_{kn}\}$, $k\in[m]$,
\begin{align}
\frac{d}{dt} u_i^n (t) =& f(u_i^n(t), t)\notag\\
& + \sum^{m}_{k=1} \frac{1}{n \alpha_{kn}} \sum^{n}_{j=1}
w_{ij}^{kn} D_k( u_j^n(t) - u_i^n(t)),\quad
i \in [n],
\label{eqn:dsk}
\end{align}
where for each $k \in [m]$, $G_{kn}=\langle V(G_{kn}),E(G_{kn}),W(G_{kn})\rangle$
 represents a sequence of dense or sparse, directed or undirected,
 and deterministic weighted or random simple graphs with the weighted matrices
\begin{equation*}
(W(G_{kn}))_{ij}=
\begin{cases}
w_{ij}^{kn} & \mbox{if $(i,j)\in E(G_{kn})$};\\
0 &\rm{otherwise},
\end{cases}
\end{equation*}
and $\alpha_{kn}$ is a scaling factor as in \eqref{eqn:ds1}.
We assume
that $D_k(u)$, $k\in[m]$, are also bounded and Lipschitz continuous.
The system~\eqref{eqn:dsk} can be derived via phase reduction \cite{BMH04}
 from general coupled oscillators, e.g.,
\begin{equation}
\dot{\xi}_i=F_0(\xi_i)+\sum_{k=1}^m\sum_{j=1}^nF_k(\xi_i-\xi_j),\quad
\xi_i\in\Rset^{\nu},\quad
i=1,\ldots,n,
\label{eqn:gco}
\end{equation}
where $F_0:\Rset^{\nu}\to\Rset^\nu$, $F_k:\Rset^\nu\to\Rset^\nu$
like \eqref{eqn:ds1} if each of them exhibits an attracting limit cycle,
 since it is expressed as a superposition of single coupled oscillator networks.
Actually, if the single oscillator
\[
\dot{\xi}=F_0(\xi),\quad
\xi\in\Rset^\nu,
\]
has an attracting limit cycle $\Gamma$,
 then we introduce a scalar phase variable $\theta(\xi)\in\Sset^1$
 for a neighborhood of $\Gamma=\{\xi=\xi^\Gamma(\theta)\mid\theta\in\Sset^1\}$
 to rewrite \eqref{eqn:gco} as
\[
\dot{\theta}_i
=\omega+\sum_{k=1}^m\sum_{j=1}^n\frac{\partial\theta}{\partial\xi}(\xi)F_k(\xi_i-\xi_j),
\]
where
\[
\omega=\frac{\partial\theta}{\partial\xi}(\xi)F_0(\xi)
\]
(see Section~2 of \cite{BMH04}).
Substituting the relation $\xi_j=\xi^\Gamma(\theta)$ yields coupled oscillators
 of the form \eqref{eqn:dsk}.
See also \cite{ST15} for such a derivation for two networks.
Similar but different extensions of coupled oscillators
 from single networks to multi-layer networks were also made
 and numerically studied recently \cite{AAPA17,KS19,MTB20}.
Our result may be extended to show that
 appropriate continuum limits related to the models are also valid.

Extending arguments given in \cite{KM17,M19},
 we prove that the IVP of the continuum limit corresponding to \eqref{eqn:dsk},
\begin{equation}
\frac{\partial}{\partial t} u(t,x) = f(u(t,x),t) + \sum^{m}_{k=1}
\int _{I} W_k(x,y) D_{k}(u(t,y)-u(t,x)) dy,
\label{eqn:csk}
\end{equation}
with the initial condition
\begin{equation}
u(x,0)=g(x),
\label{eqn:icc}
\end{equation}
where $W_k(x,y)$, $k\in[n]$, are $L^2$ functions on $I^2$,
 and $g(x)$ is an $L^2$ function on $I$,
 has a unique solution under general conditions
 and that the solution becomes the limit of those to the IVP of \eqref{eqn:dsk}
 with the initial condition
\begin{equation}
u_i^n(0) = u_{i0}^n:=n\int_{I_i^n}g(x)dx
\label{eqn:icd}
\end{equation}
in some adequate meaning as $n \to \infty$
 (see Theorems~\ref{thm:main1} and \ref{thm:main2} below).
Here
\[
I_i^n:=\begin{cases}
[(i-1)/n,i/n) & \mbox{for $i<n$};\\
[(n-1)/n,1] & \mbox{for $i=n$}.
\end{cases}
\]
Moreover, we show that
 if solutions to \eqref{eqn:dsk} are stable  or asymptotically stable for $n>0$ sufficiently large,
 then so is the corresponding solution to \eqref{eqn:csk},
 and that if solutions to \eqref{eqn:csk} are asymptotically stable,
 then  so is the corresponding solution to \eqref{eqn:dsk} in some weak meaning as $n\to\infty$
 (see Theorems~\ref{thm:main3} and \ref{thm:main4} below).

Replacing $m$ with $m+1$ and letting $w_{ij}^{m+1,n}=\omega_i^n$, $D_{m+1}(u)=1$
 and $\alpha_{n,m+1}=1$,
 we see that Eq.~\eqref{eqn:dsk} contains
\begin{align}
\frac{d}{dt} u_i^n (t) =& f ( u_i^n(t), t )+\omega_i^n\notag\\
& + \sum^{m}_{k=1} \frac{1}{n \alpha_{kn}} \sum^{n}_{j=1}
w_{ij}^{kn} D_k ( u_j^n(t) - u_i^n(t) ),\quad
i \in [n],
\label{eqn:dskw}
\end{align}
as a special case and the corresponding continuum limit is given by
\begin{equation}
\frac{\partial}{\partial t} u(t,x) = f(u(t,x),t) +
\omega(x)+\sum^{m}_{k=1}\int _{I} W_k(x,y) D_{k}(u(t,y)-u(t,x)) dy,
\label{eqn:cskw}
\end{equation}
where $\omega(x)$ is an $L^2$ function on $I$ such that
\begin{equation}
\omega_i^n=n\int_{I_i^n}\omega(x)dx, \quad i \in [n].
\label{eqn:omega}
\end{equation}
Moreover, we can treat the case
 in which the natural frequencies are randomly determined
 by $\omega_i^n=\omega^\ast$ or $0$ with some probabilities depending on $i\in[n]$
 where $\omega^\ast\neq 0$ is a constant,
 like $w_{ij}^{kn}$, $k\in[m]$ (cf. Section~2).
Thus, the results for \eqref{eqn:dsk} and \eqref{eqn:csk}
 can be applied to coupled oscillator networks with multiple frequencies
 by regarding the frequencies as a weight matrix of another graph
 although the results of \cite{KM17,M19} cannot
 even if $m=1$, i.e., depending on a single network.
Note that in \eqref{eqn:dsk}
 the natural frequencies can depend on the nodes even if they are randomly determined
 although in most of the previous research, e.g.,
 \cite{AS06,C15,CM19a,CM19b,CMM18,CN11,K84,KB02,OA08,RLZ13,R08,SA15,SA16},
 the natural frequencies are typically independent of the nodes
 and randomly determined when they are not constant.

Our theory is applicable to various types of coupled oscillator networks
 from deterministic and random ones to their mixtures.
So we choose the following three variants of the Kuramoto model \eqref{eqn:Km}
 as relatively basic ones to illustrate the theory:
\begin{itemize}
\setlength{\leftskip}{-2.5em}
\item
It has multiple natural frequencies and depends on a single graph;
\item
it is the same as the above but subjected to feedback control;
\item
it has no natural frequencies but depends on two graphs.
\end{itemize}
We also give numerical simulation results for each example.
In the last example,
 where a complete and nearest neighbor graphs are chosen more concretely as the two graphs,
 we see that a modification of type \eqref{eqn:ds2} can control
 the Kuramoto model \eqref{eqn:Km} with $\omega=0$ on the complete graph
 from a complete synchronized state to a different synchronized one.
Further applications will be reported in subsequent work.

The outline of this paper is as follows:
In Section 2, we give our four theorems:
The first one is for unique existence of solutions in the IVP
 of the continuum limit \eqref{eqn:csk} with \eqref{eqn:icc},
 the second one is for convergence of solutions to the IVP
 of the coupled oscillator network \eqref{eqn:dsk}
 with \eqref{eqn:icd} to those to the IVP of \eqref{eqn:csk} with \eqref{eqn:icc},
 and third and fourth ones are for relations of solutions to \eqref{eqn:dsk} and \eqref{eqn:csk} on stability.
Proofs of the first and second results are given in Appendices~A and B, respectively,
 while the third and fourth ones are proved there.
In the remaining three sections
 we demonstrate the theoretical results for the three variants
 of the Kuramoto model \eqref{eqn:Km} along with numerical simulations.


\section{
Theory}

In this section we give the four theorems
 for unique existence of solutions in the IVP of the continuum limit \eqref{eqn:csk},
 for convergence of solutions to the IVP of the coupled oscillator network \eqref{eqn:dsk}
 to those to the IVP of \eqref{eqn:csk},
 and for stability of solutions to \eqref{eqn:dsk} and \eqref{eqn:csk}.
Henceforth, we assume for the $L^2$ functions $W_k(x,y)$, $k\in[m]$,
that there exist positive constants $C_j$, $j=1,2$, such that
\begin{equation}
\esssup\displaylimits_{y \in I} \int_I|W_k (x,y)|dx\leq C_1
\label{eqn:assumpx}
\end{equation}
and
\begin{equation}
\esssup\displaylimits_{x\in I} \int_I|W_k (x,y)|dy \leq C_2,
\label{eqn:assumpy}
\end{equation}
where $C_j$, $j=1,2$, are independent of $k\in[m]$.
If $W_k(x,y)$, $k\in[m]$, are symmetric,
then conditions~\eqref{eqn:assumpx} and \eqref{eqn:assumpy} are equivalent.

We begin with the IVP for the continuum model \eqref{eqn:csk} with \eqref{eqn:icc}.
Let $\mathbf{u}:\Rset\to L^2(I)$ stand for an $L^2(I)$-valued function.
Extending arguments in the proof of Theorem~3.1 of \cite{KM17},
we prove the following theorem.

\begin{thm}
\label{thm:main1}
Let $f(u,t)$ and $D_k(u)$, $k\in[m]$, be Lipschitz continuous in $u$
 and let $f(u,t)$ continuous in $t$, as stated in Section~$1$.
Suppose that $W_k \in L^2(I^2)$, $k \in [m]$, satisfy \eqref{eqn:assumpy}
and $g \in L^2(I)$.
Then there is a unique solution $\mathbf{u}(t)\in C^1(\Rset;L^2(I))$
 to the IVP of \eqref{eqn:csk} with \eqref{eqn:icc}.
Moreover, the solution depends continuously on $g$.
\end{thm}

See Appendix~A for the proof of Theorem~\ref{thm:main1}.

\begin{rmk}
\label{rmk:2a}
Theorem~$3.1$ of Kaliuzhnyi-Verbovetskyi and Medvedev {\rm\cite{KM17}}
 contains the statement of Theorem~$\ref{thm:main1}$ with $m=1$ when
\[
\int_I W_1(x,y)dy=1
\]
instead of \eqref{eqn:assumpy}.
For the reader's convenience,
 we give a detailed proof of Theorem~$\ref{thm:main1}$ in Appendix~$A$.
\end{rmk}

We turn to the issue on convergence of solutions in the coupled oscillator network \eqref{eqn:dsk}
 to those in the continuum limit \eqref{eqn:csk}.
Following the approach of \cite{M19} basically
 and using a measurable function $W_k$ called a \emph{graphon} \cite{L12},
we define the asymptotic structure of the graphs $\{G_{kn}\}$ for each $k\in[m]$.
Discretize the interval $[0,1]$
by points $x_j^n = j/n$, $j \in \{0\} \cup [n]$,
so that $I_i^n= (x_{n,i-1}, x_i^n]$, $i \in [n]$.
Here the uniform mesh $\left\{x_i^n, i = 0, 1, \ldots, n \right\}$
has been chosen for simplicity, as this is sufficient for our applications,
although any other dense mesh of $[0,1]$ can be used.
Depending on whether the graph is random or deterministic,
dense or sparse, and directed or undirected,
the sequences of graphs $\{G_{kn}\}$ are constructed for each $k\in[m]$ as follows:
\begin{itemize}
\setlength{\leftskip}{-2.5em}
\item
If $G_{kn}$, $n\in\Nset$, are deterministic dense graphs, then
\begin{equation}
w_{ij}^{kn} = \langle W_k \rangle_{ij}^n
:= n^2 \int_{I_i^n \times I_j^n}W_k(x,y) dxdy.
\label{eqn:ddd}
\end{equation}
\item
If $G_{kn}$, $n\in\Nset$, are random dense graphs,
then $w_{ij}^{kn}=1$ with probability
\begin{equation}
\mathbb{P}(j \rightarrow i) = \langle W_k \rangle_{ij}^n,
\label{eqn:ddr}
\end{equation}
where the range of $W_k$ is contained in $I$.
\item
If $G_{kn}$, $n\in\Nset$, are random sparse graphs,
then $w_{ij}^{kn}=1$ with probability
\begin{equation}
\mathbb{P}(j \rightarrow i) = \alpha_{kn} \langle \tilde{W}_{kn} \rangle_{ij}^n,
\quad \tilde{W}_{kn}(x,y) :=\alpha^{-1}_{kn} \wedge W_k(x,y),
\label{eqn:sdr}
\end{equation}
where $W_k$ is a nonnegative function
and $\{\alpha_{kn}\}_{n\in\Nset}$ is a sequence such that $\alpha_{kn}\in(0,1)$, and
$\alpha_{kn}\searrow 0$ and $n \alpha_{kn} \rightarrow \infty$ as $n \rightarrow \infty$.
Here $a\wedge b=\min(a,b)$ for $a,b\in\Rset$.
Specifically, we take $\alpha_{kn} =n^{-\gamma_k}$ with $\gamma_k\in(0,\frac{1}{2})$ below.
\end{itemize}

For the random graphs,
$w_{ij}^{kn}=1$, $k\in[m]$ and $i,j\in[n]$, are assumed
to be independent Bernoulli random variables
with probability of success given by \eqref{eqn:ddr} or \eqref{eqn:sdr}.
We easily see that random sparse graphs with $\alpha_{kn}=1$ and $W_k(I^2)\subset I$
reduce to random dense graphs.
If $W_k$ is symmetric,
then the graph $G_{kn}$ is undirected (i.e., $w_{ij}^{kn}=w_{ji}^{kn}$)
and `$\to$' is replaced with `$\sim$' in \eqref{eqn:ddr} and \eqref{eqn:sdr}.
We allow a mixture of graphs of any type stated above in \eqref{eqn:dsk}.

Let
\[
d^+_{kn,i}=\sum^{n}_{j=1}w_{ij}^{kn}
\quad\mbox{and}\quad
d^-_{kn,j}=\sum^{n}_{j=1}w_{ij}^{kn},
\]
which represent the sum of the weights of directed edges
pointing to $i\in[n]$ and going from $j\in[n]$, respectively.
We call $d^+_{kn,i}$ and $d^-_{kn,j}$ \emph{in-degree} of $i$ and \emph{out-degree} of $j$.
Note that $d^+_{kn,i}=d^-_{kn,i}$ for any $i\in[n]$
if $G_{kn}$, $n\in\Nset$, are undirected.
From \eqref{eqn:assumpx} and \eqref{eqn:assumpy} we have
\begin{equation}
\begin{split}
&
d^+_{kn,i}=n^2\int_{I_i^n\times I}W_k(x,y)dxdy\leq nC_2,\\
&
d^-_{kn,j}=n^2\int_{I\times I_j^n}W_k(x,y)dxdy\leq nC_1
\end{split}
\label{eqn:dd}
\end{equation}
for deterministic dense graphs, and
\begin{equation}
\begin{split}
& \mathbb{E} d^{+}_{kn,i} = \sum^{n}_{j=1} \alpha_{kn} \langle \tilde{W}_{kn} \rangle_{ij}^n
=\alpha_{kn} n^2\int_{I}W_{k}(x,y)dxdy\leq n\alpha_{kn}C_2,\\
& \mathbb{E} d^{-}_{kn,j} = \sum^{n}_{j=1} \alpha_{kn} \langle \tilde{W}_{kn} \rangle_{n,ji}
= \alpha_{kn} n^2\int_{I}W_{k}(x,y) dxdy\leq n\alpha_{kn}C_1
\end{split}
\label{eqn:dr}
\end{equation}
with $n\gg 1$ for random graphs.
Condition~\eqref{eqn:dd} or \eqref{eqn:dr}
 is satisfied for dense graphs 
 and many random sparse graphs.
However, there are some important graphs which do not satisfy the condition.
See \cite{M19} for more details.

Since the right-hand-side of \eqref{eqn:dsk} is Lipschitz continuous in $u_i^n$, $i\in[n]$,
 we see by a fundamental result of ordinary differential equations
 (e.g., Theorem~2.1 of Chapter~1 of \cite{CL55})
 that the IVP of \eqref{eqn:dsk} with \eqref{eqn:icd} has a unique solution.
Given a solution $u_n(t)=(u_1^n(t), u_2^n(t), \ldots, u_n^n(t))$
to the IVP of the discrete model \eqref{eqn:dsk} with \eqref{eqn:icd},
we define an $L^2(I)$-valued function $\mathbf{u}_n:\Rset\to L^2(I)$ as
\begin{equation}
\mathbf{u}_n(t) = \sum^{n}_{i=1} u_i^n(t) \mathbf{1}_{I_i^n},
\label{eqn:un}
\end{equation}
where $\mathbf{1}_{I_j^n}$ represents the characteristic function of $I_i^n$, $i\in[n]$.
Let $\|\cdot\|$ denote the norm in $L^2(I)$.
Extending arguments in the proof of Theorem~3.1 of \cite{M19},
we prove the following theorem.

\begin{thm}
\label{thm:main2}
Suppose that the hypotheses of Theorem~$\ref{thm:main1}$ hold
 along with \eqref{eqn:assumpx}
 and that $D_k(u)$, $k\in[m]$, are bounded.
Let $\{G_{kn}\}_{n\in\Nset}$ be a sequence of graphs
 defined by the graphon $W_k\in L^{2}(I^2)$ for $k\in[m]$ as above,
 and let $\alpha_{kn} =n^{-\gamma_k}$,
where $\gamma_k = 0$ if $G_{kn}$ is a dense graph,
and $\gamma_k\in(0,\frac{1}{2})$ if $G_{kn}$ is a sparse random graph.
If $u_{n}(t)$ is the solution to the IVP
 of the discrete model \eqref{eqn:dsk} with \eqref{eqn:icd},
 then for any $T > 0$ we have
\[
\lim_{n \rightarrow \infty}\max_{t\in[0,T]}\|\mathbf{u}_n(t)-\mathbf{u}(t)\|=0 \quad\mathrm{a.s.},
\]
where $\mathbf{u}(t)$ represents the solution
to the IVP of the continuum limit \eqref{eqn:csk} with \eqref{eqn:icc}.
\end{thm}

See Appendix~B for the proof of Theorem~\ref{thm:main2}.

\begin{rmk}
\label{rmk:2b}
The statement of Theorem~$\ref{thm:main2}$ with $m=1$
 is the same as that of Theorem~$3.1$ of Medvedev {\rm\cite{M19}}.
For the reader's convenience, we modify his arguments in our setting
  and give a proof of Theorem~$\ref{thm:main2}$ in Appendix~B
  since necessary modifications are not so straightforward
  and some new ideas are needed.
\end{rmk}

Finally, we discuss the stability of solutions to \eqref{eqn:dsk} and \eqref{eqn:csk}.
We say that solutions $\bar{\mathbf{u}}_n(t)$ and $\bar{\mathbf{u}}(t)$
 to \eqref{eqn:dsk} and \eqref{eqn:csk} are \emph{stable}
 if for any $\epsilon>0$ there exists $\delta>0$
 such that $\|\mathbf{u}_n(t)-\bar{\mathbf{u}}_n(t)\|<\epsilon$
 and $\|\mathbf{u}(t)-\bar{\mathbf{u}}(t)\|<\epsilon$
 when $\|\mathbf{u}_n(0)-\bar{\mathbf{u}}_n(0)\|<\delta$
 and $\|\mathbf{u}(0)-\bar{\mathbf{u}}(0)\|<\delta$, respectively.
Moreover, they are called \emph{asymptotically stable}
 if $\lim_{t\to\infty}\|\mathbf{u}_n(t)-\bar{\mathbf{u}}_n(t)\|=0$
 or $\lim_{t\to\infty}\|\mathbf{u}(t)-\bar{\mathbf{u}}(t)\|=0$ additionally.
We have the following result.

\begin{thm}
\label{thm:main3}
Suppose that under the hypotheses of Theorem~$\ref{thm:main2}$,
 the discrete model \eqref{eqn:dsk} and continuum limit \eqref{eqn:csk}, respectively,
 have solutions $\bar{\mathbf{u}}_n(t)$ and $\bar{\mathbf{u}}(t)$ such that
\[
\lim_{n\to\infty}\|\bar{\mathbf{u}}_n(t)-\bar{\mathbf{u}}(t)\|=0\quad\text{a.s.}
\]
for any $t\in[0,\infty)$.
Then the following hold$\,:$
\begin{enumerate}
\setlength{\leftskip}{-1.8em}
\item[\rm(i)]
If $\bar{\mathbf{u}}_n(t)$ is stable $($resp. asymptotically stable$)$ a.s. for $n>0$ sufficiently large,
 then $\mathbf{u}(t)$ is also stable $($resp. asymptotically stable$);$
\item[\rm(ii)]
If $\bar{\mathbf{u}}(t)$ is asymptotically stable, then
\begin{equation}
\lim_{t\to\infty}\lim_{n\to\infty}\|\mathbf{u}_n(t)-\bar{\mathbf{u}}_n(t)\|=0
\quad\text{a.s.},
\label{eqn:thm3}
\end{equation}
where $\mathbf{u}_n(t)$ is any solution to \eqref{eqn:dsk}
 such that $\mathbf{u}_n(0)$ is contained in the basin of attraction for $\bar{\mathbf{u}}(t)$.
\end{enumerate}
\end{thm}

\begin{proof}
Let $\epsilon>0$ be sufficiently small.
We begin with part~(i).

Suppose that $\bar{\mathbf{u}}_n(t)$ is stable a.s. for $n>0$ sufficiently large
 but $\bar{\mathbf{u}}(t)$ is not.
When $n>0$ is sufficiently large, we have
\begin{equation}
\|\bar{\mathbf{u}}_n(t)-\bar{\mathbf{u}}(t)\|<\epsilon\quad\text{a.s.}
\label{eqn:thm3a}
\end{equation}
for any $t\in[0,\infty)$.
We choose $n,T>0$ sufficiently large and take $\mathbf{u}(0)=\mathbf{u}_n(0)$ such that 
\[
\|\mathbf{u}_n(\tau)-\bar{\mathbf{u}}_n(\tau)\|<\epsilon\quad\text{a.s.},\quad
\|\mathbf{u}(\tau)-\bar{\mathbf{u}}(\tau)\|>3\epsilon
\]
and
\begin{equation}
\|\mathbf{u}_n(\tau)-\mathbf{u}(\tau)\|<\epsilon\quad\text{a.s.}
\label{eqn:thm3b}
\end{equation}
for some $\tau\in(0,T)$.
Here condition \eqref{eqn:thm3b}  is guaranteed by Theorem~\ref{thm:main2}.
Hence,
\begin{align*}
&
\|\mathbf{u}_n(\tau)-\mathbf{u}(\tau)\|\\
&
\ge\|\mathbf{u}(\tau)-\bar{\mathbf{u}}(\tau)\|
 -\|\mathbf{u}_n(\tau)-\bar{\mathbf{u}}_n(\tau)\|
 -\|\bar{\mathbf{u}}_n(\tau)-\bar{\mathbf{u}}(\tau)\|>\epsilon\quad\text{a.s.},
\end{align*}
which contradicts \eqref{eqn:thm3b}.
Thus, if $\bar{\mathbf{u}}_n(t)$ is stable a.s. for $n>0$ sufficiently large,
 then so is $\bar{\mathbf{u}}(t)$.

Suppose that $\bar{\mathbf{u}}_n(t)$ is asymptotically stable a.s.
 for $n>0$ sufficiently large but $\bar{\mathbf{u}}(t)$ is not.
We take $\mathbf{u}(0)=\mathbf{u}_n(0)$
 such that $\mathbf{u}_n(0)$ is contained
 in the basin of attraction for $\bar{\mathbf{u}}_n(t)$ a.s.,
 and choose $n,T>0$ sufficiently large such that
\[
\|\mathbf{u}_n(T)-\bar{\mathbf{u}}_n(T)\|<\epsilon\quad\text{a.s.},\quad
\|\mathbf{u}(T)-\bar{\mathbf{u}}(T)\|>3\epsilon
\]
and
\begin{equation}
\|\mathbf{u}_n(T)-\mathbf{u}(T)\|<\epsilon\quad\text{a.s.}
\label{eqn:thm3c}
\end{equation}
Since $\|\bar{\mathbf{u}}_n(T)-\bar{\mathbf{u}}(T)\|<\epsilon$ a.s.
 by \eqref{eqn:thm3a},  we have
\begin{align*}
&
\|\mathbf{u}_n(T)-\mathbf{u}(T)\|\\
&
\ge\|\mathbf{u}(T)-\bar{\mathbf{u}}(T)\|
 -\|\mathbf{u}_n(T)-\bar{\mathbf{u}}_n(T)\|
 -\|\bar{\mathbf{u}}_n(T)-\bar{\mathbf{u}}(T)\|>\epsilon\quad\text{a.s.},
\end{align*}
which contradicts \eqref{eqn:thm3c}.
Thus, we obtain part~(i).

We turn to part~(ii).
Suppose that $\bar{\mathbf{u}}(t)$ is asymptotically stable.
When $n>0$ is sufficiently large, we have \eqref{eqn:thm3a} for any $t\in[0,\infty)$.
We take $\mathbf{u}(0)=\mathbf{u}_n(0)$
 such that $\mathbf{u}(0)$ is contained in the basin of attraction for $\bar{\mathbf{u}}(t)$,
 and choose $n,T>0$ sufficiently large such that
\[
\|\mathbf{u}(T)-\bar{\mathbf{u}}(T)\|<\epsilon
\]
and condition~\eqref{eqn:thm3b} holds with $\tau=T$.
Hence,
\begin{align*}
&
\|\mathbf{u}_n(T)-\bar{\mathbf{u}}_n(T)\|\\
&
\le\|\mathbf{u}_n(T)-\mathbf{u}(T)\|
 +\|\mathbf{u}(T)-\bar{\mathbf{u}}(T)\|
 +\|\bar{\mathbf{u}}(T)-\bar{\mathbf{u}}_n(T)\|<3\epsilon\quad\text{a.s.}
\end{align*}
This means part~(ii).
\end{proof}

\begin{rmk}
\label{rmk:2c}
By the contrapositive of Theorem~$\ref{thm:main3}$, under its hypotheses,
 if $\bar{\mathbf{u}}(t)$ is unstable $($resp. not asymptotically stable$)$ in \eqref{eqn:csk},
 then so is $\bar{\mathbf{u}}_n(t)$ in \eqref{eqn:dsk} a.s. for $n>0$ sufficiently large.
On the other hand, if condition~\eqref{eqn:thm3} does not hold,
 then $\bar{\mathbf{u}}(t)$ is not asymptotically stable in \eqref{eqn:csk}.
Note that condition~\eqref{eqn:thm3} holds
 if $\bar{\mathbf{u}}_n(t)$ is asymptotically stable a.s. for $n>0$ sufficiently large
 but it may hold even if not.
Actually, it holds if for any $\epsilon>0$ there exists $T_0>0$ such that for $T>T_0$
\[
\lim_{n\to\infty}\|\mathbf{u}_n(T)-\bar{\mathbf{u}}_n(T)\|<\epsilon\quad\text{a.s.}
\]
even when $\lim_{n\to\infty}\lim_{t\to\infty}\|\mathbf{u}_n(t)-\bar{\mathbf{u}}_n(t)\|\neq 0$ a.s.
\end{rmk}

In examples of Sections~3 and 5,
 we assume that $f(u,t)\equiv 0$, so that Theorem~\ref{thm:main3} does not apply as seen below.
Using modifying the above arguments slightly,
 we can extend the result to such a situation.

Assume that $f(u,t)\equiv 0$.
For $\theta\in\Rset$,
 let $\boldsymbol{\theta}$ represent the constant function $u=\theta$.
If $\bar{\mathbf{u}}_n(t)$ is a solution to the discrete model \eqref{eqn:dsk},
 then so is $\bar{\mathbf{u}}_n(t)+\boldsymbol{\theta}$ for any $\theta\in\Rset$.
Similarly, if $\bar{\mathbf{u}}(t)$ is a solution to the continuum limit \eqref{eqn:csk},
 then so is $\bar{\mathbf{u}}(t)+\boldsymbol{\theta}$ for any $\theta\in\Rset$.
Let $\U_n=\{\bar{\mathbf{u}}_n(t)+\boldsymbol{\theta}\mid\theta\in\Rset\}$
 and $\U=\{\bar{\mathbf{u}}(t)+\boldsymbol{\theta}\mid\theta\in\Rset\}$
 denote the families of solutions to \eqref{eqn:dsk} and \eqref{eqn:csk}, respectively.
Recall that $\U_n$ (resp. $\U$) is called \emph{stable}
 if solutions starting in its (smaller) neighborhood
 remain in its (larger) neighborhood for $t\ge 0$,
 and \emph{asymptotically stable} if $\U_n$ (resp. $\U$) is stable
 and the distance between such solutions and $\U_n$ (resp. $\U$) converges to zero
 as $t\to\infty$.

\begin{thm}
\label{thm:main4}
Suppose that the hypotheses of Theorem~$\ref{thm:main3}$ hold and $f(u,t)\equiv 0$.
Then the following hold$\,:$
\begin{enumerate}
\setlength{\leftskip}{-1.8em}
\item[\rm(i)]
If $\U_n$ is stable $($resp. asymptotically stable$)$ a.s. for $n>0$ sufficiently large,
 then $\U$ is also stable $($resp. asymptotically stable$);$
\item[\rm(ii)]
If $\U$ is asymptotically stable, then
\begin{equation}
\lim_{t\to\infty}\lim_{n\to\infty}\inf_{\theta\in\Rset}
 \|\mathbf{u}_n(t)-\bar{\mathbf{u}}_n(t)-\boldsymbol\theta\|=0\quad\text{a.s.},
\label{eqn:thm4}
\end{equation}
where $\mathbf{u}_n(t)$ is any solution to \eqref{eqn:dsk}
 such that $\mathbf{u}_n(0)$ is contained in the basin of attraction for $\U$.
\end{enumerate}
\end{thm}

\begin{proof}
We proceed as in the proof of Theorem~\ref{thm:main3} with some modifications.
Let $\epsilon>0$ be sufficiently small.
We begin with part~(i).

Suppose that $\U_n$ is stable a.s. for  $n>0$ sufficiently large but $\U$ is not.
We choose $n,T>0$ sufficiently large and take $\mathbf{u}(0)=\mathbf{u}_n(0)$
 such that for some $\theta_n\in\Rset$ and any $\theta\in\Rset$
\[
\|\mathbf{u}_n(\tau)-\bar{\mathbf{u}}_n(\tau)-\boldsymbol{\theta}_n\|<\epsilon
 \quad\text{a.s.},\quad
\|\mathbf{u}(\tau)-\bar{\mathbf{u}}(\tau)-\boldsymbol{\theta}\|>3\epsilon
\]
and
\begin{equation}
\|\mathbf{u}_n(\tau)-\mathbf{u}(\tau)\|<\epsilon\quad\text{a.s.}
\label{eqn:thm4a}
\end{equation}
for some $\tau\in(0,T)$.
Since $\|\bar{\mathbf{u}}_n(\tau)-\bar{\mathbf{u}}(\tau)\|<\epsilon$ a.s.
 by \eqref{eqn:thm3a}, we have
\begin{align*}
\|\mathbf{u}_n(\tau)-\mathbf{u}(\tau)\|
\ge&\|\mathbf{u}(\tau)-\bar{\mathbf{u}}(\tau)-\boldsymbol{\theta}_n\|\\
& -\|\mathbf{u}_n(\tau)-\bar{\mathbf{u}}_n(\tau)-\boldsymbol{\theta}_n\|
 -\|\bar{\mathbf{u}}_n(\tau)-\bar{\mathbf{u}}(\tau)\|>\epsilon\quad\text{a.s.},
\end{align*}
which contradicts \eqref{eqn:thm4a}.
Thus, if $\U_n$ is stable a.s. for  $n>0$ sufficiently large, then so is $\U$.

Suppose that $\U_n$ is asymptotically stable a.s. for $n>0$ sufficiently large
 but $\U$ is not.
We take $\mathbf{u}(0)=\mathbf{u}_n(0)$
 such that $\mathbf{u}_n(0)$ is contained in the basin of attraction for $\U_n$ a.s.,
 and choose $n,T>0$ sufficiently large such that for some $\theta_n\in\Rset$
 and any $\theta\in\Rset$
\[
\|\mathbf{u}_n(T)-\bar{\mathbf{u}}_n(T)-\boldsymbol{\theta}_n\|<\epsilon
 \quad\text{a.s.},\quad
\|\mathbf{u}(T)-\bar{\mathbf{u}}(T)-\boldsymbol{\theta}\|>3\epsilon
\]
and
\begin{equation}
\|\mathbf{u}_n(T)-\mathbf{u}(T)\|<\epsilon\quad\text{a.s.}
\label{eqn:thm4b}
\end{equation}
Since $\|\bar{\mathbf{u}}_n(T)-\bar{\mathbf{u}}(T)\|<\epsilon$ a.s.
 by \eqref{eqn:thm3a},
 for some $\theta_n\in\Rset$ we have
\begin{align*}
\|\mathbf{u}_n(T)-\mathbf{u}(T)\|
\ge&\|\mathbf{u}(T)-\bar{\mathbf{u}}(T)-\boldsymbol{\theta}_n\|\\
& -\|\mathbf{u}_n(T)-\bar{\mathbf{u}}_n(T)-\boldsymbol{\theta}_n\|
 -\|\bar{\mathbf{u}}_n(T)-\bar{\mathbf{u}}(T)\|>\epsilon\quad\text{a.s.},
\end{align*}
which yields a contradiction.
Thus, we obtain part~(i).

We turn to part~(ii).
Suppose that $\U$ is asymptotically stable.
When $n>0$ is sufficiently large, we have \eqref{eqn:thm4a} for any $\tau\in[0,\infty)$.
We take $\mathbf{u}(0)=\mathbf{u}_n(0)$
 such that $\mathbf{u}(0)$ is contained in the basin of attraction for $\U$,
 and choose $n,T>0$ sufficiently large such that
\[
\|\mathbf{u}(T)-\bar{\mathbf{u}}(T)-\boldsymbol{\theta}\|<\epsilon\quad\text{a.s.}
\]
for some $\theta\in\Rset$ and condition~\eqref{eqn:thm4b} holds.
Hence, for some $\theta\in\Rset$
\begin{align*}
\|\mathbf{u}_n(T)-\bar{\mathbf{u}}_n(T)-\boldsymbol{\theta}\|
\le&\|\mathbf{u}_n(T)-\mathbf{u}(T)\|
 +\|\mathbf{u}(T)-\bar{\mathbf{u}}(T)-\boldsymbol{\theta}\|\\
&
 +\|\bar{\mathbf{u}}(T)-\bar{\mathbf{u}}_n(T)-\boldsymbol{\theta}\|<3\epsilon
  \quad\text{a.s.}
\end{align*}
This means part~(ii).
\end{proof}

\begin{rmk}\
\label{rmk:2d}
\begin{enumerate}
\setlength{\leftskip}{-1.8em}
\item[\rm(i)]
Under the hypotheses of Theorem~$\ref{thm:main4}$, we see that
 if $\U$ is unstable $($resp. not asymptotically stable$)$ in \eqref{eqn:csk},
 then so is $\U_n$ in \eqref{eqn:dsk} a.s. for $n>0$ sufficiently large,
 and that if condition~\eqref{eqn:thm4} does not hold,
 then $\U$ is not asymptotically stable in \eqref{eqn:csk},
 as in Remark~$\ref{rmk:2c}$.
\item[\rm(ii)]
In examples of Sections~$3$-$5$,
 the state variables $u_i$, $i\in[n]$, belong to $\Sset^1$ as well as the range of $u$,
 and so does the parameter $\theta$.
\end{enumerate}
\end{rmk}

Henceforth we assume that $W_k(x,y)$, $k\in[m]$, is nonnegative
without loss of generality.
Actually, we can always write $W_k(x,y)=W_k^+(x,y)-W_k^-(x,y)$
for some nonnegative functions $W_k^+(x,y),W_k^-(x,y)$,
so that Eq.~\eqref{eqn:csk} becomes
\begin{align*}
\frac{\partial}{\partial t} u(t,x)
=& f(u(t,x),t)+\sum^{m}_{k=1} \int _{I} W_k^+(x,y) D_{k}(u(t,y)-u(t,x)) dy\\
& -\sum^{m}_{k=1} \int _{I} W_k^-(x,y) D_{k}(u(t,y)-u(t,x)) dy.
\end{align*}
Similarly, we also assume that
any element $w_{ij}^{kn}$ of the weight matrix $W(G_{kn})$
is nonnegative for $k\in[m]$, $n\in\Nset$ and $i,j\in[n]$ in \eqref{eqn:dsk}.
We remark that such a treatment
requires analysis of coupled oscillator networks depending on multiple graphs,
contrary to a statement given at the end of Section 2 of \cite{M19}.


\section{Kuramoto model with multiple natural frequencies}
Our first example is the Kuramoto model with multiple natural frequencies,
\begin{equation}
\frac{d}{dt} u_i^n (t) = \omega_i^n
+ \frac{1}{n \alpha_n} \sum^{n}_{j=1}
w_{ij}^n \sin \left( u_j^n(t) - u_i^n(t) \right),\quad i \in [n],
\label{eqn:dex1}
\end{equation}
and its continuum limit
\begin{equation}
\frac{\partial}{\partial t}u(t,x) = \omega(x)
+ \int_I W(x,y) \sin( u(t,y)-u(t,x) ) dy , \quad x \in I,
\label{eqn:cex1}
\end{equation}
which, respectively, correspond to \eqref{eqn:dskw} and \eqref{eqn:cskw}
with $m=1$, $f(u,t)=0$ and $D(u)=\sin u$.
Here the dependence of the graph on the index $k$ is dropped out.
We also assume that the graphon $W(x,y)$ has the form
\begin{equation}
W(x,y) = H_1(x)H_2(y),
\label{con:1}
\end{equation}
where $H_\ell:I\to\Rset$, 
$\ell=1,2$, are bounded measurable functions and $H_1>0$ on $I$.
The graphon \eqref{con:1} is simple and of rank $1$ \cite{L12}.

\subsection{General case}
We begin with the continuum limit \eqref{eqn:cex1}.
Assume that there exists a constant $C\neq 0$ such that
\begin{equation}
C=\int_I H_2(y)\cos\left(\arcsin\left(\frac{\omega(y)-\Omega}{C H_1(y)}\right)\right)dy.
\label{con:c}
\end{equation}
Then we easily show that the continuum limit \eqref{eqn:cex1} has synchronized solutions given by
\begin{equation}
u(t,x)=\Omega t+U(x)+\theta,\quad
U(x)=\arcsin \left(\frac{\omega(x) - \Omega}{C H_1(x)}\right),
\label{eqn:solcex1}
\end{equation}
where $\theta\in\Sset^1$ is any constant,
the range of the $\arcsin$ function is $[-\frac{1}{2}\pi,\frac{1}{2}\pi]$ and
\begin{align}
\Omega=\int_I \frac{H_2(x) \omega(x)}{H_1(x)} dx
\bigg/\int_I \frac{H_2(x)}{H_1(x)} dx.
\label{eqn:ccex1}
\end{align}
Actually, substituting \eqref{con:1} and \eqref{eqn:solcex1}
into the right-hand-side of \eqref{eqn:cex1}, we obtain
\begin{align}
&
\omega(x)
+ H_1(x) \cos U(x)\,\int_I H_2(y) \sin U(y)dy\notag\\
&\qquad
- H_1(x) \sin U(x)\,\int_I H_2(y) \cos U(y)dy=\Omega
\label{eqn:rcex1}
\end{align}
since
\begin{equation}
\int_I H_2(y)\sin U(y)\,dy
=\frac{1}{C}\int_I \frac{H_2(y)}{H_1(y)}(\omega(y)-\Omega)dy=0,
\label{eqn:con2}
\end{equation}
where we have used the relation \eqref{eqn:ccex1}.
This means that Eq.~\eqref{eqn:solcex1} is a solution to \eqref{eqn:cex1}
for any $\theta\in\Sset^1$.
Note that the continuum limit \eqref{eqn:cex1}
 may have a different solution from \eqref{eqn:solcex1}.

We discuss the linear stability
of the synchronized solutions \eqref{eqn:solcex1} to \eqref{eqn:cex1}.
The linearized equation of \eqref{eqn:cex1} around \eqref{eqn:solcex1} is written as
\[
\frac{\partial}{\partial t}u(t,x)=\L u(t,x),
\]
where $\L:L^2(I)\to L^2(I)$ is the linear operator given by
\begin{align}
\L\phi (x)
=&\int_I H_1(x)H_2(y)(\phi(y)-\phi(x))\cos(U(y)-U(x))dy\notag\\
=& H_1(x) \cos U(x)\int_I H_2(y) \phi(y) \cos U(y) dy \notag\\
& +H_1(x) \sin U(x)\int_I H_2(y) \phi(y) \sin U(y) dy \notag\\
& -H_1(x) \phi(x)\cos U(x)\int_I H_2(y) \cos U(y) dy.
\label{eqn:lcex1}
\end{align}
Here the relation \eqref{eqn:con2} has been used.
Note that the linear operator $\L$ is independent of $\theta\in\Sset^1$.

Assume that $\L\phi(x)=0$ for any $x\in\Sset^1$.
Using \eqref{con:c} and \eqref{eqn:lcex1}, we have
\begin{equation}
\phi(x)=c_1+c_2\tan U(x)
\label{eqn:phi}
\end{equation}
since $H_1(x)>0$ on $I$,
 where $c_1,c_2$ are constants given by
\[
c_1=\frac{1}{C}\int_I H_2(y) \phi(y) \cos U(y)dy,\quad
c_2=\frac{1}{C}\int_I H_2(y) \phi(y) \sin U(y)dy=0.
\]
Substituting \eqref{eqn:phi} into \eqref{eqn:lcex1}
 and using \eqref{eqn:con2}, we obtain
\[
c_2\int_I H_2(y)\left( \frac{1}{\cos U(y)} - 2 \cos U(y) \right) dy=0.
\]
Thus, we see that $\mathrm{Ker}(\L)=\Span\{1\}$ if
\begin{equation}
\int_I H_2(y)
\left( \frac{1}{\cos U(y)} - 2 \cos U(y) \right) dy\neq 0,
\label{con:ex}
\end{equation}
and $\mathrm{Ker}(\L)=\Span\{1,\tan U(x)\}$ otherwise.
Note that the eigenspace $\Span\{1\}$ for the zero eigenvalue
 comes from the rotational symmetry, and that
\[
\int_I H_2(y) \left( \frac{1}{\cos U(y)} - 2 \cos U(y) \right) dy < 0
\]
if $H_2(y) \geq 0$ and $|U(x)| \leq \pi/4$.

Suppose that the graphon $W\not\equiv 0$ is nonnegative,
 i.e., $H_1(x),H_2(x)\geq 0$ for any $x\in I$.
We compute
\begin{align}
\langle \L \phi, \phi \rangle
&= \int_{I^2} W(x,y) (\phi(y) - \phi(x)) \phi(x)\cos (U(y)-U(x)) dydx \notag\\
&= \int_{I^2} W(x,y) \phi(y) \phi(x) \cos (U(y)-U(x)) dydx \notag\\
& \quad
- \int_{I^2} W(x,y) \phi(x)^2 \cos (U(y)-U(x)) dydx \notag\\
&= \int_{I^2} W(x,y) \phi(y) \phi(x) \cos (U(y)-U(x)) dxdy \notag\\
& \quad
- \frac{1}{2} \int_{I^2} W(x,y) (\phi(x)^2 + \phi(y)^2) \cos (U(y)-U(x)) dxdy \notag\\
&= - \frac{1}{2} \int_{I^2} W(x,y) (\phi(y)-\phi(x))^2 \cos (U(y)-U(x)) dxdy,
\label{eqn:lcex1a}
\end{align}
where $\langle\cdot,\cdot\rangle$ represents the inner product in $L^2(I)$.
Hence, if
\begin{equation}
|U(x)| \leq \frac{\pi}{4}\quad \mbox{for any $x\in I$},
\label{con:2}
\end{equation}
then $\langle \L \phi, \phi \rangle\le 0$ for any $\phi\in L^2(I)$,
so that the synchronized solutions \eqref{eqn:solcex1} are linearly stable.
Moreover, if $H_2(x)\geq 0$ as well as condition \eqref{con:2} holds,
then the family of synchronized solutions is asymptotically stable
since $\langle \L \phi, \phi \rangle<0$ for $\phi\not\in\Span\{1\}$.

We turn to the Kuramoto model \eqref{eqn:dex1}
when $G_n$ is deterministic and dense.
We have $\alpha_n=1$ and
\begin{equation}
w_{ij}^n = h_{1n,i} h_{2n,j},
\label{con:1d}
\end{equation}
where
\[
h_{\ell n,i}=n\int_{I_i^n}H_\ell(x)dx,\quad
\ell=1,2.
\]
We see that Eq.~\eqref{eqn:dex1} has synchronized solutions given by
\begin{equation}
u_i^n(t)= \Omega_D t+U_i+\theta,\quad
U_i=\arcsin\left( \frac{\omega_i^n - \Omega_D}{C_D h_{1n,i}} \right),
\label{eqn:soldex1}
\end{equation}
where $\theta\in\Sset^1$ is any constant and
\[
\Omega_D=\sum^{n}_{i=1} \frac{h_{2n,i} \omega_i^n}{h_{1n,i}}
\bigg/\sum^{n}_{i=1} \frac{h_{2n,i}}{h_{1n,i}},\quad
C_D=\frac{1}{n} \sum^{n}_{i=1} h_{2n,i} \cos U_i.
\]
Actually, substituting \eqref{con:1d} and \eqref{eqn:soldex1}
into the right-hand-side of \eqref{eqn:dex1}, we have
\begin{align*}
&
\omega_i^n+
\frac{1}{n}\biggl(h_{1n,i}\cos U_i(t)\sum^{n}_{j=1}h_{2n,j}\sin U_j(t)\\
&\qquad
-h_{1n,i}\sin U_i(t)\sum^{n}_{j=1}h_{2n,j}\cos U_j(t)\biggr)=\Omega_D
\end{align*}
since
\begin{align*}
&
\sum^{n}_{i=1} h_{2n,i} \sin U_i
=\frac{1}{C_D}\sum_{i=1}^n\frac{h_{2n,i}}{h_{1n,i}}(\omega_i^n-\Omega_D)=0,\\
&
\frac{1}{n}h_{1n,i}\sin U_i(t)\sum^{n}_{j=1}h_{2n,j}\cos U_j(t)
=\omega_i^n-\Omega_D.
\end{align*}
This means that for any $\theta\in\Sset^1$
Eq.~\eqref{eqn:soldex1} is a solution to \eqref{eqn:dex1}
which exists if and only if for some constant $C_D$
\[
C_D=\frac{1}{n}\sum^{n}_{i=1}h_{2n,i}
\cos\left(\arcsin\left( \frac{\omega_i^n - \Omega_D}{C_D h_{1n,i}} \right) \right)
\]
like \eqref{con:c}.
We also easily see that as $n\to\infty$,
\begin{equation}
U_i\to U(x)\quad\mbox{and}\quad
C_D\to C
\label{eqn:limitex1}
\end{equation}
if $i/n\to x$.
Thus, by \eqref{eqn:limitex1}, the solution \eqref{eqn:soldex1} of \eqref{eqn:dex1}
 converges to the solution \eqref{eqn:solcex1} of \eqref{eqn:cex1}.
Note that the Kuramoto model \eqref{eqn:dex1}
 may have a different of solution from \eqref{eqn:soldex1}.

We discuss the linear stability of the synchronized solutions \eqref{eqn:soldex1} to \eqref{eqn:dex1}.
The Jacobian matirix of \eqref{eqn:dex1} for \eqref{eqn:soldex1} is given by
\begin{equation*}
A=
\begin{pmatrix}
-\displaystyle\sum_{j\neq 1} \frac{w_{1j}^n\cos (U_j - U_1)}{n}
& \displaystyle\frac{w_{12}^n \cos (U_2 -U_1)}{n}
& \displaystyle\cdots & \displaystyle\frac{w_{1n}^n \cos (U_n - U_1)}{n} \\
\displaystyle\frac{w_{21}^n \cos (U_1 - U_2)}{n}
& -\displaystyle\sum_{j\neq 2} \frac{w_{2j}^n \cos (U_j - U_2)}{n}
& \cdots & \displaystyle\frac{w_{2n}^n \cos (U_n - U_2)}{n} \\
\vdots & \displaystyle\vdots & \displaystyle\ddots
& \vdots \\
\displaystyle\frac{w_{n1}^n \cos (U_1 - U_n)}{n}
& \displaystyle\frac{w_{n2}^n \cos (U_2 -U_n)}{n}
& \cdots
& -\displaystyle\sum_{j\neq n} \frac{w_{nj}^n \cos (U_j - U_n)}{n}
\end{pmatrix}.
\end{equation*}
Using Gershgorin's theorem (see, e.g., Theorem~7.2.1 of \cite{GL13}), we have
\[
\biggl|\lambda+\sum_{j\ne i}\frac{w_{ij}^n \cos (U_j - U_i)}{n}\biggr|
\le\sum_{j\ne i}\biggl|\frac{w_{ij}^n\cos (U_j - U_i)}{n}\biggr|,\quad
i\in[n],
\]
where $\lambda$ is any eigenvalue of $A$.
Hence, if
\begin{equation}
|U_i| \leq \frac{\pi}{4},\quad
i \in [n],
\label{con:2d}
\end{equation}
then $A$ has no eigenvalue with a positive real part,
so that the synchronized solutions \eqref{eqn:soldex1} are linearly stable.
Moreover, as in \eqref{eqn:lcex1},
we see that $\mathrm{Ker}(A)=\Span\{\mathbf{1}\}$ if
\begin{equation}
\sum^{n}_{i=1} h_{2n,i} \left( \frac{1}{\cos U_i} - 2 \cos U_i \right) \neq 0,
\label{con:2d'}
\end{equation}
which holds when condition~\eqref{con:2d} holds and $h_{2n,i}\ge 0$ for $i\in[n]$,
 and $\mathrm{Ker}(A)=\Span\{\mathbf{1},\mathbf{e}\}$ otherwise, where
\[
\mathbf{1}=\left(1, 1,\cdots,1\right)^{\mathrm{T}},\quad
\mathbf{e}=\left( \frac{\sin U_1}{\cos U_1}, \frac{\sin U_2}{\cos U_2},
\cdots, \frac{\sin U_n}{\cos U_n} \right)^{\mathrm{T}}.
\]
Thus, if conditions~\eqref{con:2d} and \eqref{con:2d'} hold
as well as $w_{ij}^n \geq 0$ for any $i,j\in[n]$ and $w_{ij}^n>0$ for some $i,j\in[n]$,
then the family of synchronized solutions given by \eqref{eqn:soldex1} is asymptotically stable
since the zero eigenvalue only comes from the rotational symmetry.

From Theorem~\ref{thm:main4} we see that
 if the family of solutions given by \eqref{eqn:soldex1} in \eqref{eqn:dex1}
 is asymptotically stable for $n>0$ stable,
 then so is the family of solutions given by \eqref{eqn:solcex1} in \eqref{eqn:cex1}, and that
 if the family of solutions \eqref{eqn:solcex1} is asymptotically stable,
 then the family of solutions \eqref{eqn:soldex1} satisfies \eqref{eqn:thm4}.
Note that condition~\eqref{eqn:thm4} holds
 if the family of solutions \eqref{eqn:soldex1} is asymptotically stable for $n>0$ sufficiently large.
Thus, the above results for the asymptotic stability of the families of solutions
 \eqref{eqn:solcex1} and \eqref{eqn:soldex1} consist with Theorem~\ref{thm:main4}.

\subsection{Simple three cases}

We concretely set
\begin{equation}
\mbox{$H_1(x)=p,H_2(x)=1$ (i.e., $W(x,y)=p$) and $\omega(x)=a(x-\frac{1}{2})$},
\label{eqn:ex11}
\end{equation}
where $a>0$ and $p\in(0,1]$ are constants,
and treat three cases in which the graph $G_n$ is undirected and deterministic dense,
random dense or random sparse.
We first discuss the continuum limit \eqref{eqn:cex1}.
We have $\Omega=0$ in \eqref{eqn:ccex1},
so that the function $U(x)$ in \eqref{eqn:solcex1} is expressed as
\begin{equation}
U(x)=\arcsin \left( \frac{a \left(x - \frac{1}{2}\right)}{pC} \right),
\label{eqn:ucex11}
\end{equation}
where
\begin{align}
C=&\int_I \cos\left( \arcsin \left( \frac{a \left(x - \frac{1}{2}\right)}{pC} \right) \right) dx\notag\\
=& \frac{pC}{a} \left(\arcsin\left(\frac{a}{2pC} \right)
+\frac{a}{2pC} \sqrt{1- \left( \frac{a}{2pC} \right)^2 } \right).
\label{eqn:acex1'}
\end{align}
The constant $C$ is equivalent to the so-called \emph{order parameter} \cite{K84}.
We rewrite \eqref{eqn:acex1'} as
\begin{equation}
\frac{a}{p}=\arcsin\left(\frac{a}{2pC} \right)
+\frac{a}{2pC} \sqrt{1- \left( \frac{a}{2pC}\right)^2}.
\label{eqn:acex1}
\end{equation}
These expressions give the synchronized solutions \eqref{eqn:solcex1}
to the continuum limit \eqref{eqn:cex1}.
In particular, we have
\[
\max_{x\in I}|U(x)|=\arcsin\left|\frac{a}{2pC}\right|
\]
so that condition~\eqref{con:2} holds
and the family of synchronized solutions given by \eqref{eqn:solcex1}
is asymptotically stable if
\begin{equation}
\left|\frac{a}{2pC}\right|\leq \frac{1}{\sqrt{2}}.
\label{con:2'}
\end{equation}
Since $\varphi(\eta)=\arcsin\eta+\eta\sqrt{1- \eta^2}$ is monotonically increasing on $(0,1)$ and
\[
\lim_{\eta\to 0}\varphi(\eta)=0,\quad
\lim_{\eta\to 1}\varphi(\eta)=\frac{\pi}{2},
\]
we easily see that the synchronized solutions \eqref{eqn:solcex1} exist
 if and only if $a/p\in(0,\pi/2]$.
Moreover, condition~\eqref{con:2'} is equivalent to
\begin{equation}
\frac{a}{p}\le\varphi\left(\frac{1}{\sqrt{2}}\right)=\tfrac{1}{4}\pi+\tfrac{1}{2}.
\label{eqn:cona}
\end{equation}

\begin{figure}[t]
\begin{center}
\includegraphics[scale=0.7]{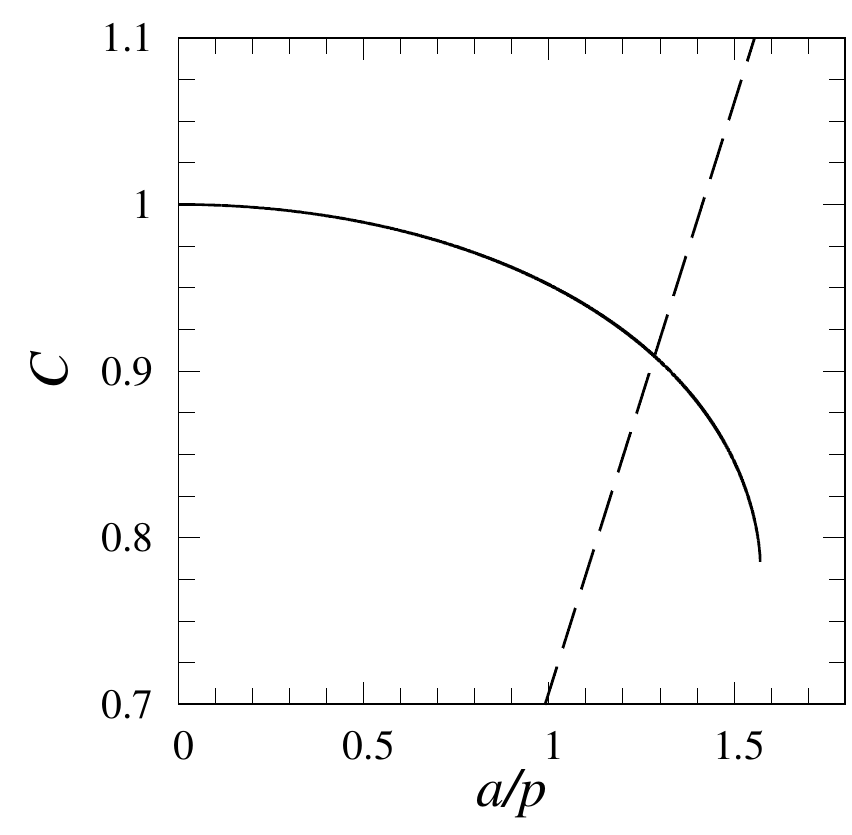}
\caption{Dependence of $C$ on $a/p$.
Condition~\eqref{con:2'} holds in the left side of the dashed line $C=(a/p)/\sqrt{2}$.}
\label{fig:5.1a}
\end{center}
\end{figure}

\begin{figure}[t]
\begin{center}
\includegraphics[scale=0.7]{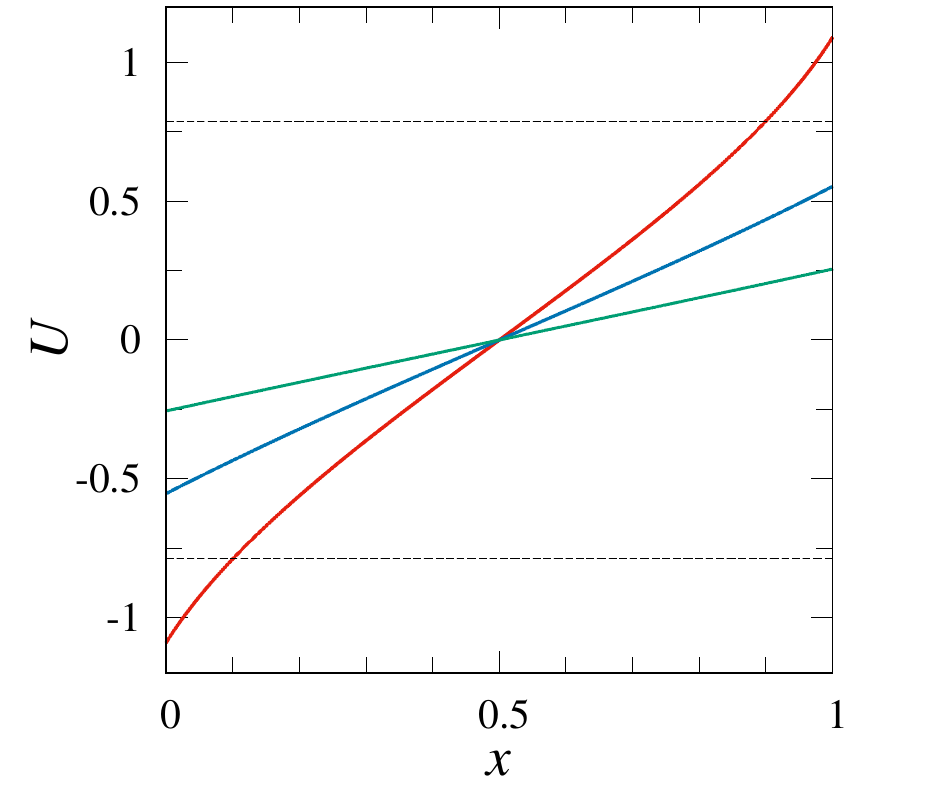}
\caption{Shape of $U(x)$ given by \eqref{eqn:ucex11}.
It is plotted as green, blue and red lines for
$a/p=0.5,1$ and $1.5$, respectively.
The dotted lines represent $U=\pm\pi/4$.}
\label{fig:5.1b}
\end{center}
\end{figure}

Figure~\ref{fig:5.1a} shows the dependence of $C$ on $a/p$,
which is easily computed by the relation \eqref{eqn:acex1}.
In the left side of the dashed line there,
condition~\eqref{con:2'} holds so that
the synchronized solutions \eqref{eqn:solcex1} are asymptotically stable in \eqref{eqn:cex1}.
In Fig.~\ref{fig:5.1b}
the shape of $U(x)$ is plotted as green, blue and red lines
for $a/p=0.5,1$ and $1.5$, respectively.
The values of $a/(2pC)$ were numerically computed
as $0.252716\ldots$, $0.525268\ldots$ and $0.886571\ldots$
for $a/p=0.5,1$ and $1.5$,
respectively.
In the figure we observe that $U(x)$ violates condition~\eqref{con:2'} for $a=1.5$
while it does not for $a/p=0.5,1$,
as predicted in Fig.~\ref{fig:5.1a}.

\subsubsection{Deterministic undirected dense graph}

We turn to the Kuramoto model \eqref{eqn:dex1}.
We begin with a deterministic undirected dense graph with
\begin{equation}
w_{ij}^n=p,\quad
i,j\in[n],
\label{eqn:ex1wa}
\end{equation}
which follows from \eqref{eqn:ddd} and \eqref{eqn:ex11}.
We carried out numerical simulations for the Kuramoto model \eqref{eqn:dex1}
with $n=1000$, $\alpha_n=1$, $p=1$ and $a=1$.
The initial values $u_i^n(0)$, $i\in[n]$, were independently randomly chosen
 according to the uniform distribution on $[-\pi,\pi]$.
This corresponds, for instance, to a situation in which the step function
\[
g(x)=u_i^n(0)\quad\mbox{for $x\in I_i^n$, $i\in[n]$,}
\]
is chosen as the initial condition of the continuum limit \eqref{eqn:cex1}.
We took the initial condition
 since we want to see whether our theoretical prediction is valid for such a general one.

\begin{figure}[t]
\begin{center}
\includegraphics[scale=0.3]{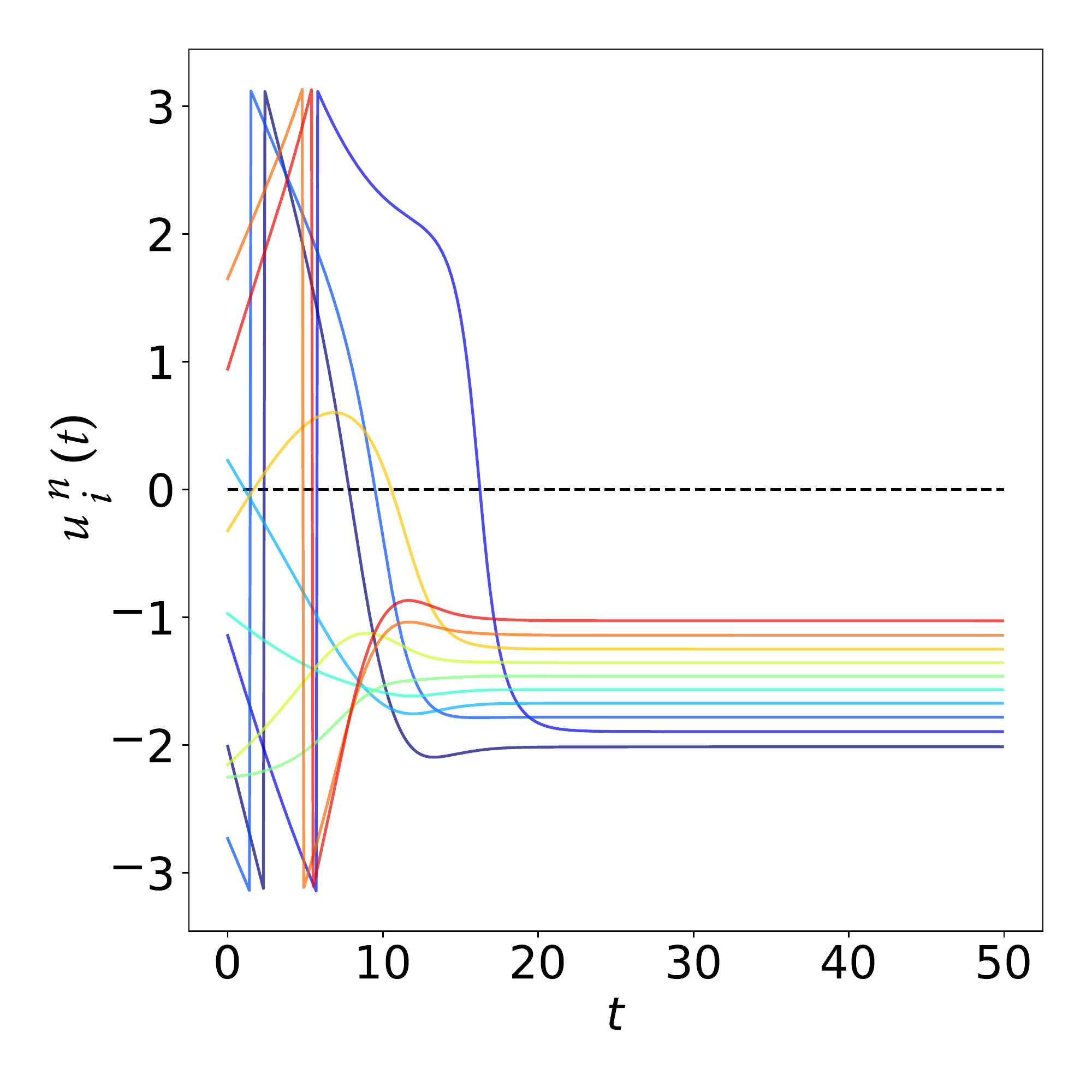}
\caption{Numerical simulation result
of the Kuramoto model \eqref{eqn:dex1} with \eqref{eqn:ex1wa}
for $n=1000$, $\alpha_n=1$, $p=1$ and $a=1$:
The time-history of every 100th node (from 100th to 1000th)
is plotted with a different color.
The nodes are lined up
from bottom to top in the synchronized state.}
\label{fig:5.1c}
\end{center}
\end{figure}

\begin{figure}[t]
\begin{center}
\includegraphics[scale=0.3]{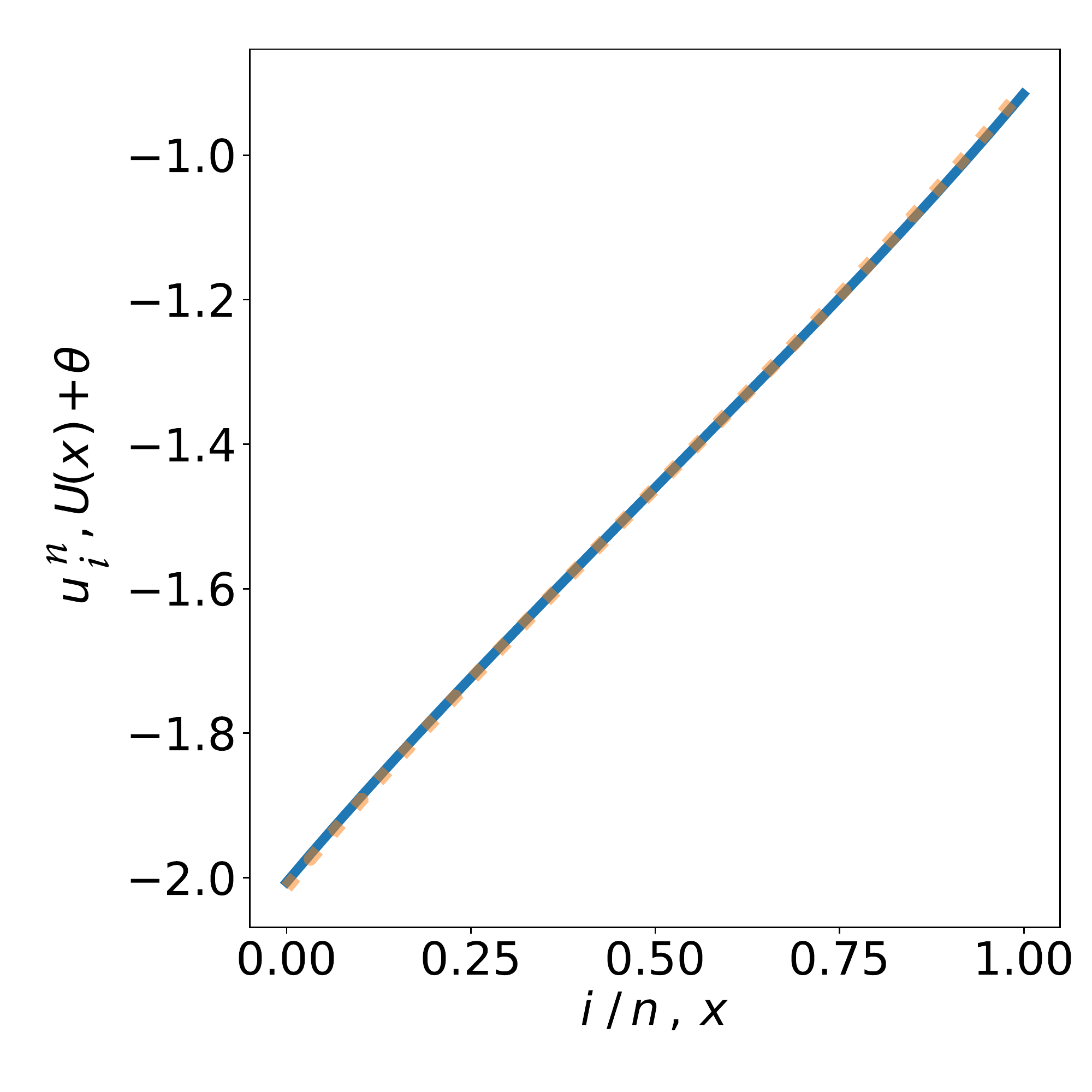}
\caption{Comparison between $u_i^n(50)$, $i\in[n]$,
in the Kuramoto model \eqref{eqn:dex1} with \eqref{eqn:ex1wa}
and the continuum limit synchronized solution \eqref{eqn:solcex1a}
with $a/(2pC)=0.525268\ldots$ and $\theta=-1.460872\ldots$
for $n=1000$, $\alpha_n=1$, $p=1$ and $a=1$.
The former and latter are plotted as orange dotted and blue solid lines, respectively. }
\label{fig:5.1d}
\end{center}
\end{figure}

Figure~\ref{fig:5.1c} shows
the time-history of every 100th node (from 100th to 1000th).
We observe that the response rapidly converges to the synchronized state,
 which is given by
\begin{equation}
u_i^n=\arcsin\biggl(\frac{a(2i-1-n)}{2npC_D}\biggr)+\theta
\label{eqn:soldex1a}
\end{equation}
with $n=1000$, $a/(2pC_D)\approx a/(2pC)=0.525268\ldots$
 from the above theory, where $\theta$ is estimated as
\[
\theta\approx\frac{1}{n}\sum_{i=1}^n u_i^n(50)=-1.460872\ldots
\]
from the numerical result.
We also remark that the response converged
to the synchronized state given by \eqref{eqn:soldex1a} for $a<\pi/2$
even when condition~\eqref{eqn:cona} does not hold.
In Fig.~\ref{fig:5.1d} the response of \eqref{eqn:dex1} at $t=50$ for $a=1$
 is compared with the continuum limit synchronized solution
\begin{equation}
u(x)=\arcsin\biggl(\frac{a(2x-1)}{2pC}\biggr)+\theta,
\label{eqn:solcex1a}
\end{equation}
to which Eq.~\eqref{eqn:soldex1a} converges with $i/n\to x$ as $n\to\infty$,
 with $a/(2pC)=0.525268\ldots$ and $\theta=-1.460872\ldots$.
We see that both coincide almost completely, as predicted theoretically.

\subsubsection{Random 
 dense graph}

\begin{figure}[t]
\begin{center}
\includegraphics[scale=0.26]{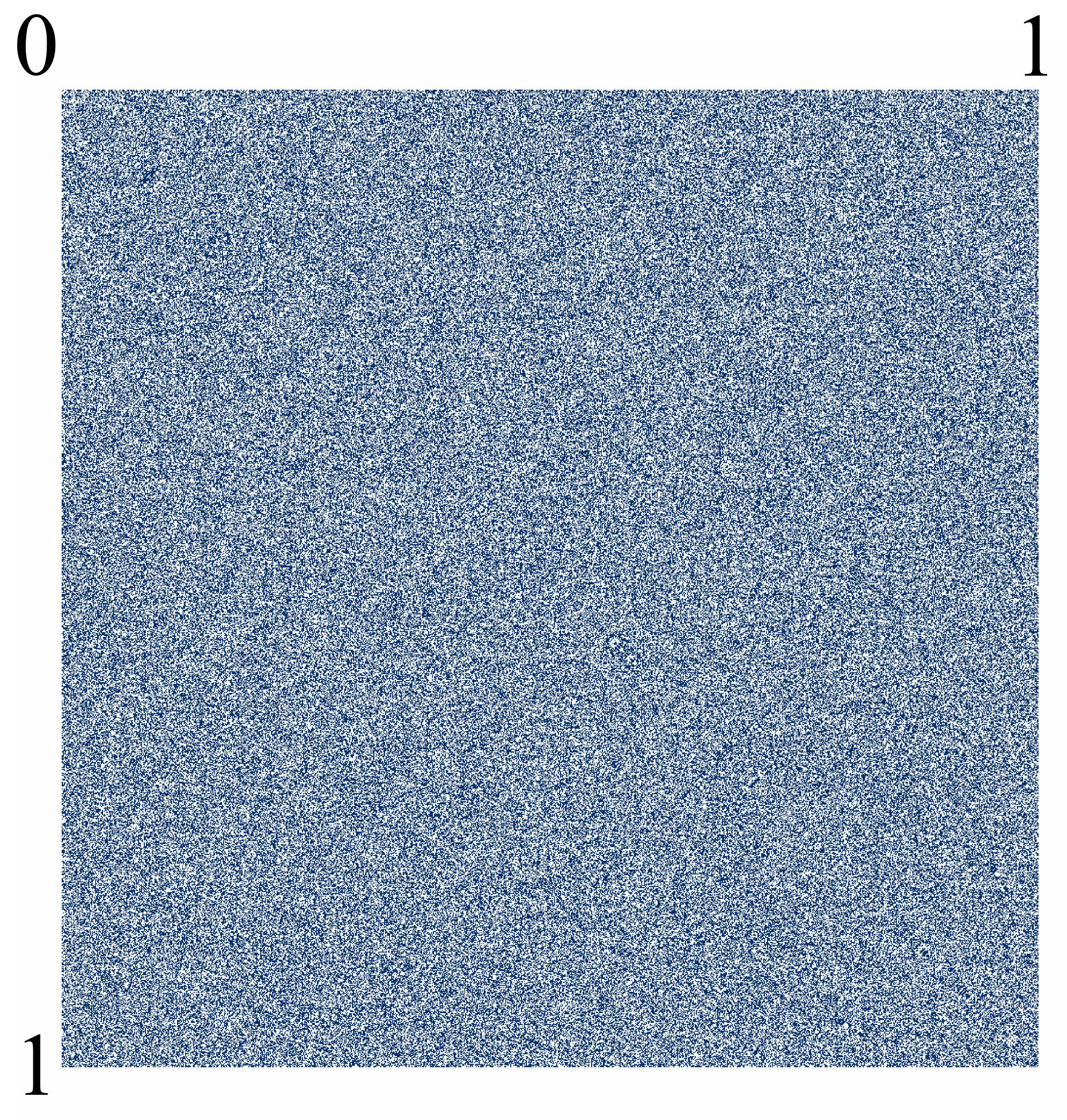}
\caption{Pixel picture of a sampled weight matrix
for the random undirected dense graph
given by $w_{ij}^n=1$, $i,j\in[n]$, with probability \eqref{eqn:ex1wb}
for $n=1000$ and $p=0.5$.
The abscissa and ordinate represent $j/n$ and $i/n$, respectively.
The color of the corresponding pixel is blue if $w_{ij}^n=1$ and it is white otherwise.}
\label{fig:5.1e}
\end{center}
\end{figure}

We next consider a random undirected dense graph given by $w_{ij}^n=1$ with probability
\begin{equation}
\Pset(j\sim i)=p,\quad
i,j\in[n],
\label{eqn:ex1wb}
\end{equation}
which follows from \eqref{eqn:ddr} and \eqref{eqn:ex11}.
Figure~\ref{fig:5.1e} represents the weight matrix
for a numerically computed sample of the random undirected dense graph
with $n=1000$ and $p=0.5$.
We carried out numerical simulations
for the Kuramoto model \eqref{eqn:dex1}
with the weight matrix displayed in Fig.~\ref{fig:5.1e}
for $n=1000$, $\alpha_n=1$ and $a=0.5$.
The initial values $u_i^n(0)$, $i\in[n]$, were independently randomly chosen on $[-\pi,\pi]$.

\begin{figure}[t]
\begin{center}
\includegraphics[scale=0.3]{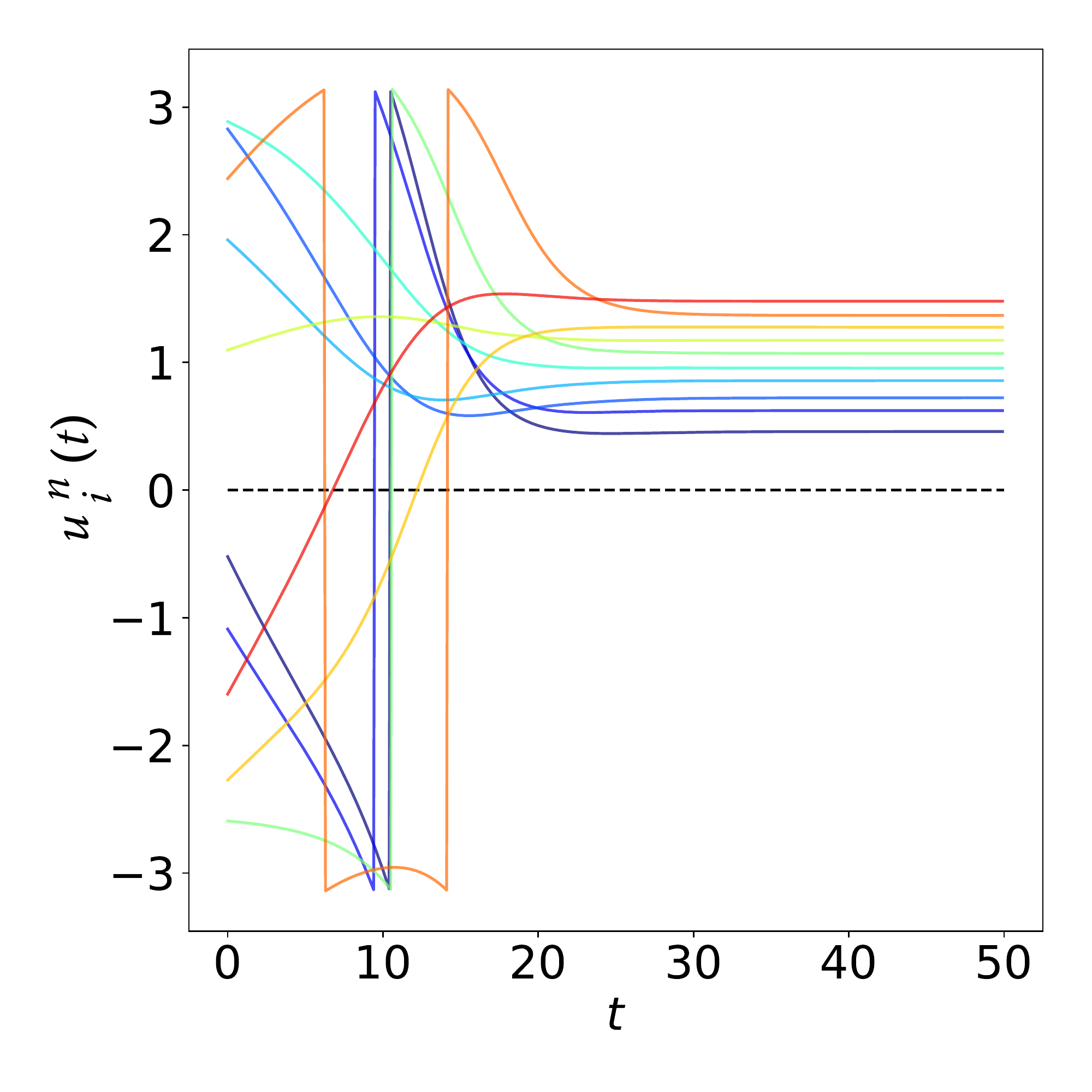}
\caption{Numerical simulation result
of the Kuramoto model \eqref{eqn:dex1}
with the weight matrix displayed in Fig.~\ref{fig:5.1e}
for $n=1000$, $\alpha_n=1$, $p=0.5$ and $a=0.5$.
See also the caption of Fig.~\ref{fig:5.1c}.}
\label{fig:5.1f}
\end{center}
\end{figure}

\begin{figure}[t]
\begin{center}
\includegraphics[scale=0.3]{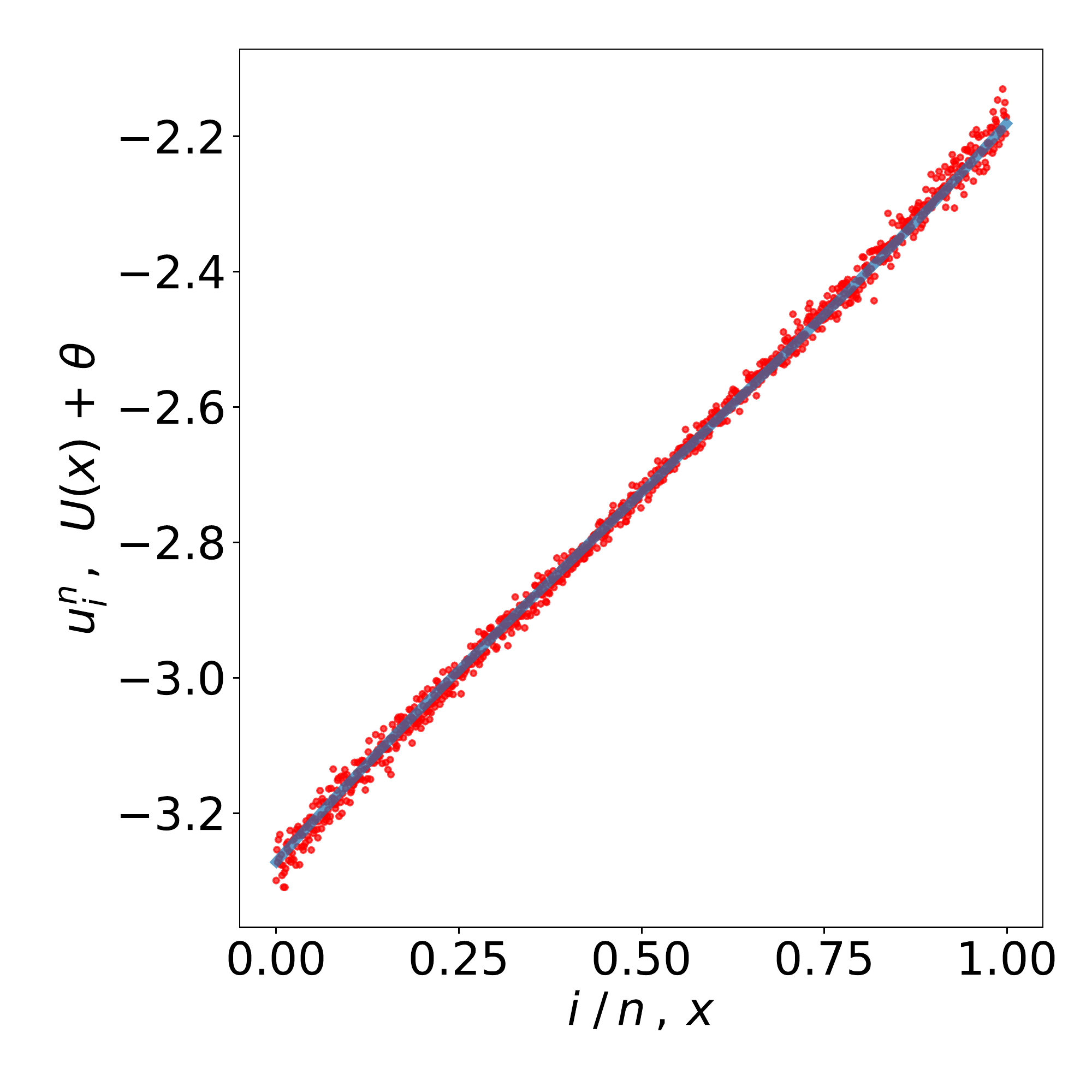}
\caption{Comparison between $u_i^n(50)$, $i\in[n]$,
 in the Kuramoto model \eqref{eqn:dex1} with \eqref{eqn:ex1wb}
 and the continuum limit synchronized solution \eqref{eqn:solcex1a}
 with $a/(2pC)=0.525268\ldots$ and $\theta=1.059373\ldots$ for $a=0.5$ and $n=1000$.
 The former and latter are plotted as small orange disks and a blue line, respectively. }
\label{fig:5.1g}
\end{center}
\end{figure}

Figure~\ref{fig:5.1f} shows the time-history of every 100th node.
We observe that the response rapidly converges to the synchronized state,
which is given by \eqref{eqn:soldex1a}
 with $a/(2pC_D)\approx a/(2pC)=0.525268\ldots$
 from the above theory, where
\[
\theta\approx{\color{black}\frac{1}{n}}\sum_{i=1}^n u_i^n(50)=1.059373\ldots
\]
from the numerical result.
In Fig.~\ref{fig:5.1g} the response of \eqref{eqn:dex1} at $t=50$ for $a=1$
 is compared with the continuum limit synchronized solution \eqref{eqn:solcex1a}
with $a/(2pC)\color{black}=0.525268\ldots$ and $\theta=1.059373\ldots$.
We see that their agreement is good, as predicted theoretically,
 although small fluctuations due to randomness are found.

\begin{figure}[t]
\begin{center}
\includegraphics[scale=0.6]{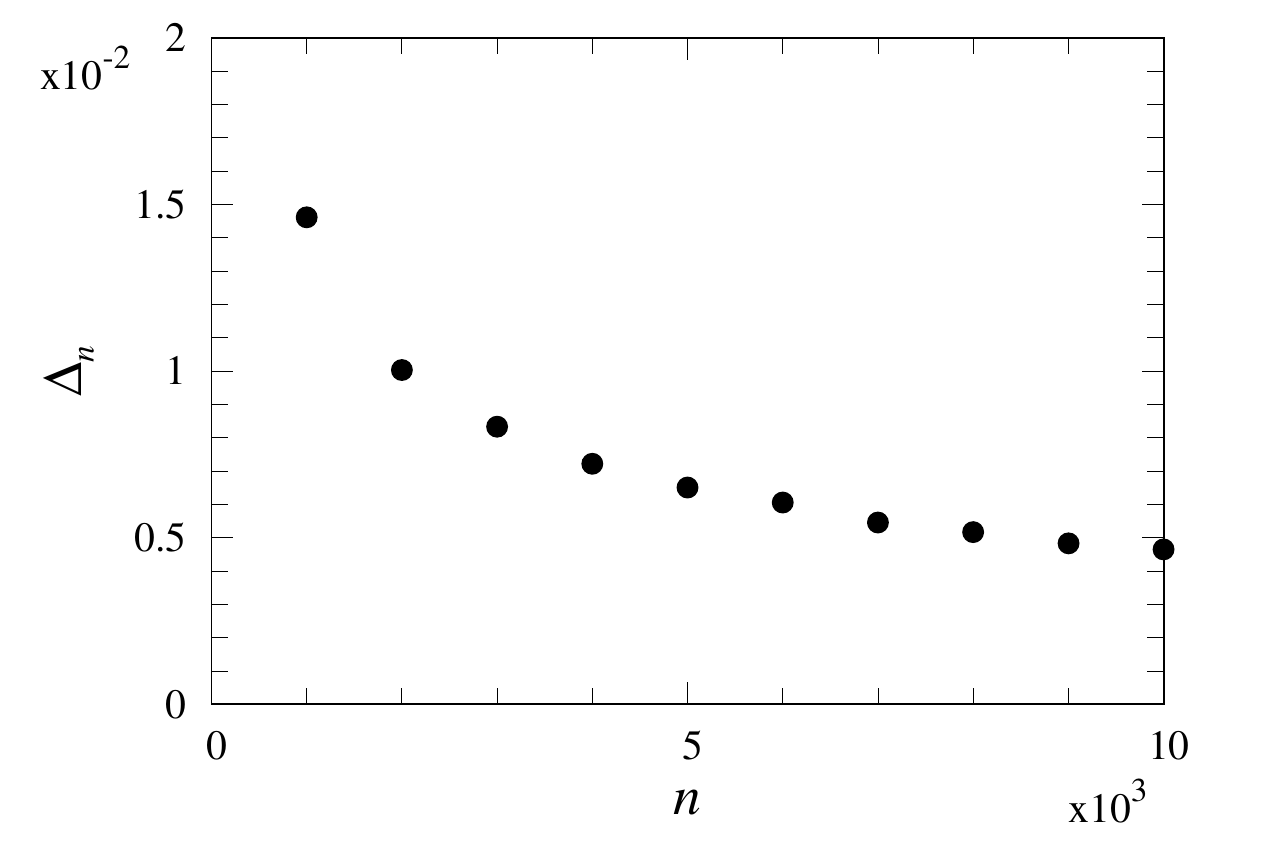}
\caption{Dependence of the approximate convergence error $\Delta_n$
 on the node number $n$ in the Kuramoto model \eqref{eqn:dex1} with \eqref{eqn:ex1wb}
 for $a=0.5$ and $T=100$. }
\label{fig:5.1g'}
\end{center}
\end{figure}

Figure~\ref{fig:5.1g'} shows how the convergence error
 of $\mathbf{u}_n(T)$ to $U(x)+\theta$
 depends on the node number $n$ for $T=100$, where
\begin{align}
\Delta_n^2=& \frac{1}{n}
 \sum_{i=1}^n\left(u_i^n(100)-U\left(\frac{2i-1}{2n}\right)\right)^2\notag\\
& -\biggl(\frac{1}{n}\sum_{i=1}^n\left(u_i^n(100)-U\left(\frac{2i-1}{2n}\right)\right)\biggr)^2,
\label{eqn:error}
\end{align}
which expresses an approximation
 for the $L^2(I)$ convergence error of $\mathbf{u}_n(T)$ as $n\to\infty$ given by
\[
\min_{\theta\in\Sset^1}\int_I \left(\mathbf{u}_n(T)-U(x)-\theta\right)^2 dx
\]
with $T>0$ sufficiently large.
Here the initial values $u_i^n(0)$, $i\in[n]$, were independently randomly chosen
according to the uniform distribution on $[-\pi,\pi]$ for each $n$.
We observe that the error $\Delta_n$ decreases  as $n$ increases
 even though different initial conditions were taken.

\subsubsection{Random sparse graph}

\begin{figure}[t]
\begin{center}
\includegraphics[scale=0.26]{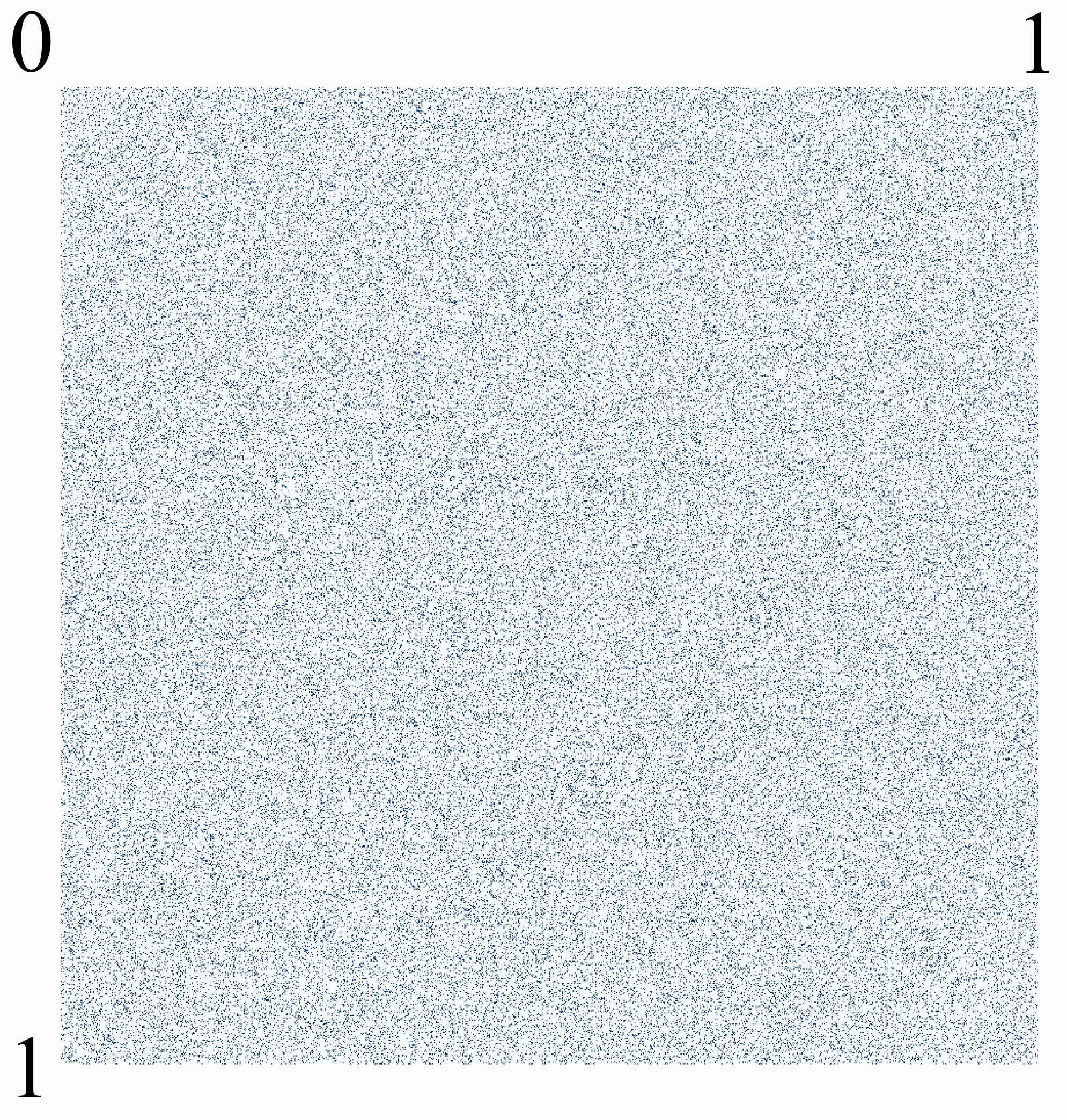}
\caption{Pixel picture of a sampled weight matrix
for the random sparse graph
given by $w_{ij}^n=1$, $i,j\in[n]$, with probability \eqref{eqn:ex1wb}
for $n=1000$ and $p=1$.
See also the caption of Fig.~\ref{fig:5.1e}.}
\label{fig:5.1h}
\end{center}
\end{figure}

We next consider a random undirected sparse graph
given by $w_{ij}^n=1$ with probability
\begin{equation}
\Pset(j\sim i)=n^{-\gamma}p,\quad
i,j\in[n]
\label{eqn:ex1wc}
\end{equation}
which follows from \eqref{eqn:sdr} and \eqref{eqn:ex11} with $\alpha_n=n^{-\gamma}$,
where $\gamma\in(0,0.5)$.
Figure~\ref{fig:5.1h} represents the weight matrix
for a numerically computed sample of the random undirected sparse graph
with $n=1000$, $p=1$ and $\gamma=0.3$, like Fig.~\ref{fig:5.1e}.
We carried out numerical simulations
for the Kuramoto model \eqref{eqn:dex1}
with the weight matrix displayed in Fig.~\ref{fig:5.1h}
for $n=1000$, $\gamma=0.3$ and $a=1$.
The initial values $u_i^n(0)$, $i\in[n]$, were independently randomly chosen on $[-\pi,\pi]$.

\begin{figure}[t]
\begin{center}
\includegraphics[scale=0.3]{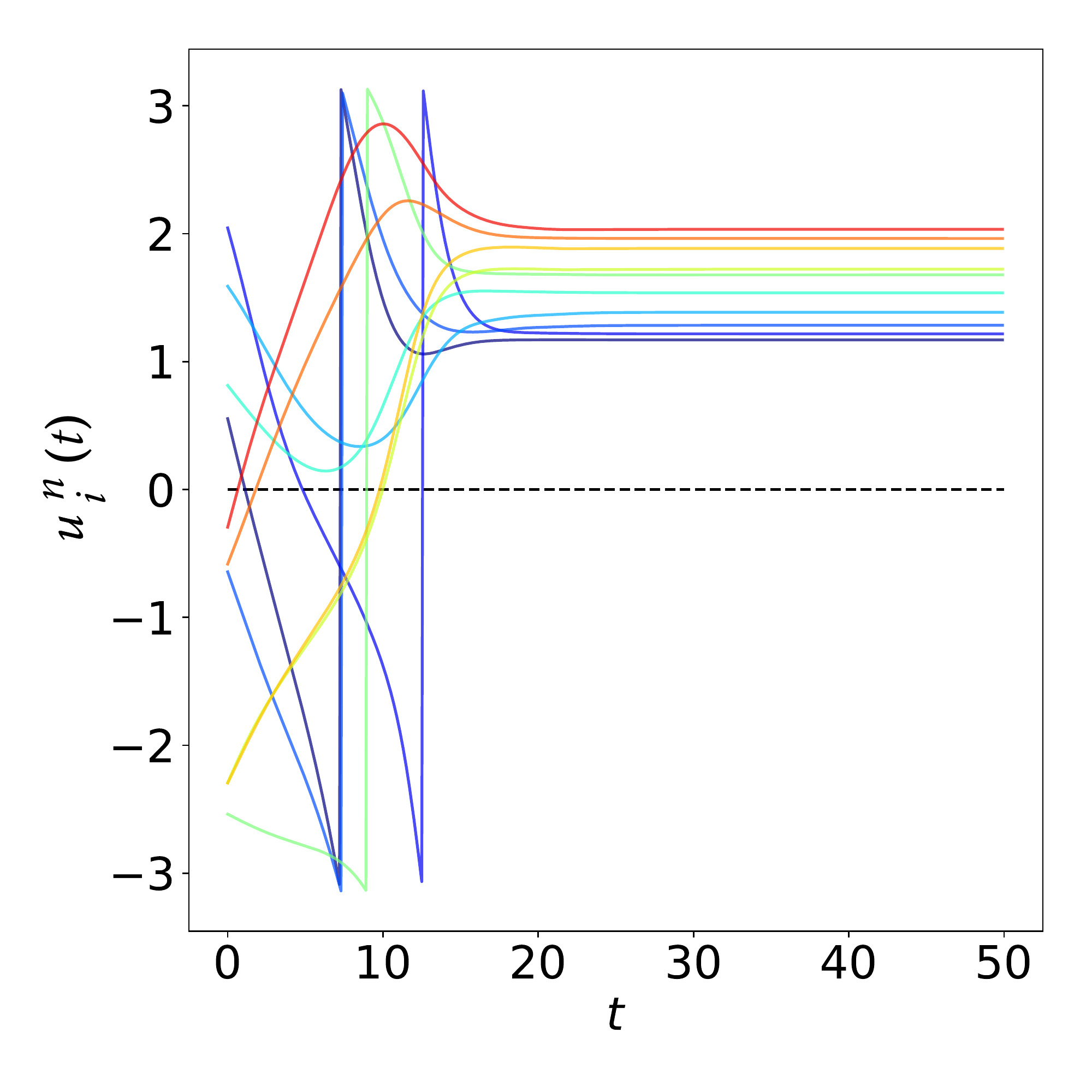}
\caption{Numerical simulation result of the Kuramoto model \eqref{eqn:dex1}
 with the weight matrix displayed in Fig.~\ref{fig:5.1h}
 for $n=1000$, $\gamma=0.3$ and $a=1$.
See also the caption of Fig.~\ref{fig:5.1c}}.
\label{fig:5.1i}
\end{center}
\end{figure}

\begin{figure}[t]
\begin{center}
\includegraphics[scale=0.3]{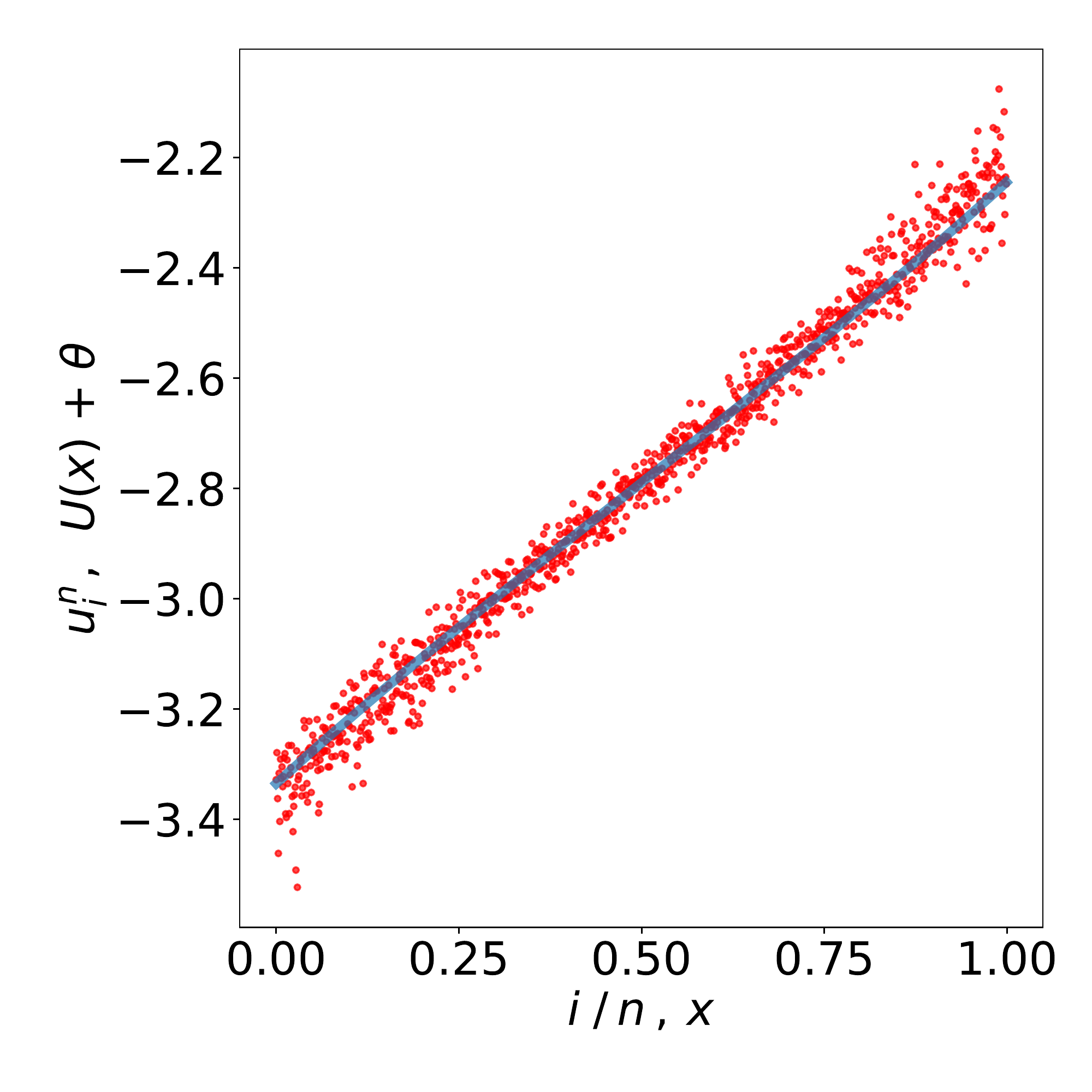}
\caption{Comparison between $u_i^n(50)$, $i\in[n]$,
in the Kuramoto model \eqref{eqn:dex1} with \eqref{eqn:ex1wc}
and the continuum limit synchronized solution \eqref{eqn:solcex1a}
with $a/(2pC)=0.525268\ldots$ and $\theta=1.640532\ldots$
 for $a=1$ and $n=1000$:
The former and latter are plotted as small orange disks and a blue line, respectively. }
\label{fig:5.1j}
\end{center}
\end{figure}

Figure~\ref{fig:5.1i} shows the time-history of every 100th node.
We observe that the response rapidly converges to the synchronized state,
which is given by \eqref{eqn:soldex1a}
with $a/(2pC_D)\approx a/(2pC)=0.525268\ldots$
 from the above theory, where
\[
\theta\approx\frac{1}{n}\sum_{i=1}^n u_i^n(50)=1.640532\ldots
\]
from the numerical result, in Fig.~\ref{fig:5.1i}.
In Fig.~\ref{fig:5.1j} the response of \eqref{eqn:dex1} at $t=50$ for $a=1$
is compared with the continuum limit synchronized solution \eqref{eqn:solcex1a}
with $a/(2pC)=0.525268\ldots$ and $\theta=1.640532\ldots$.
We see that their agreement is good, as predicted theoretically,
 although some fluctuations due to randomness are found.

\begin{figure}[t]
\begin{center}
\includegraphics[scale=0.6]{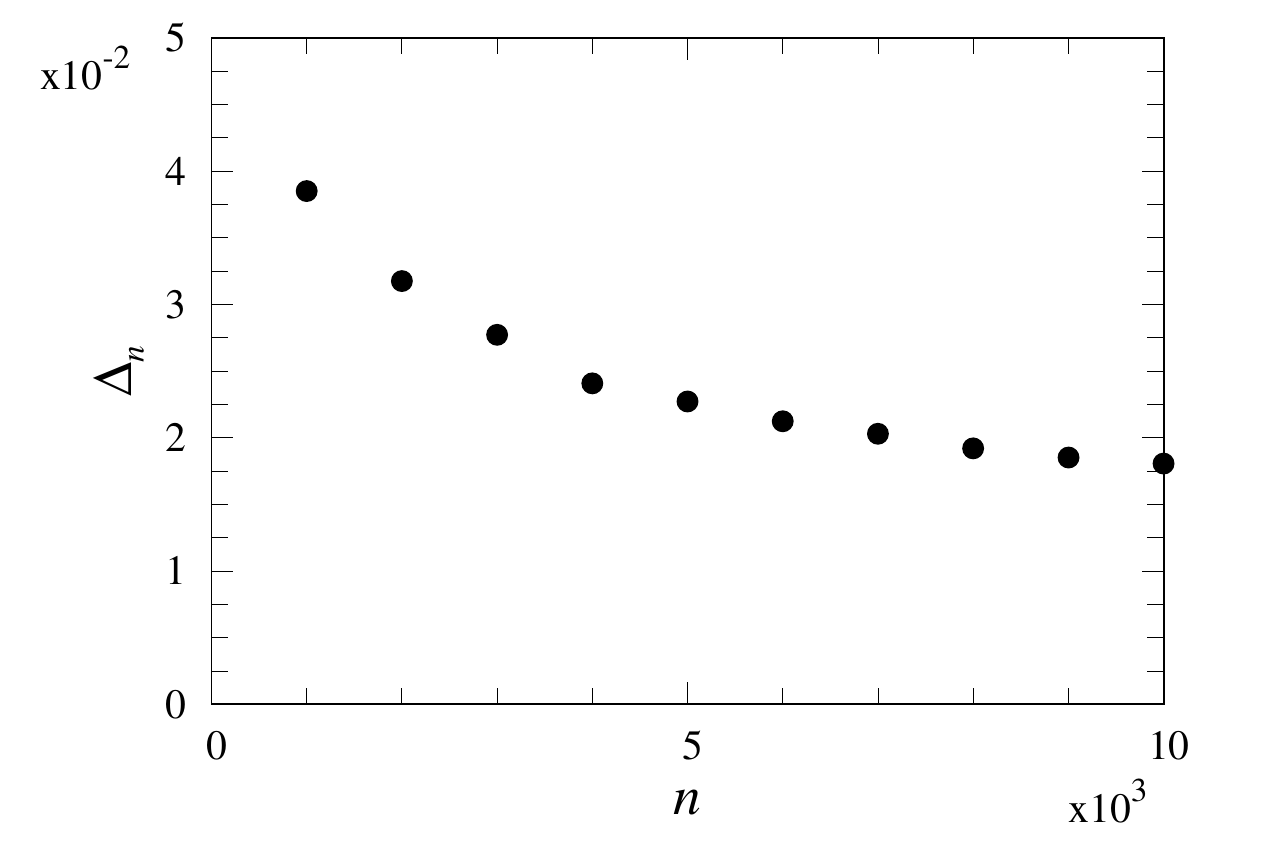}
\caption{Dependence of the approximate convergence error $\Delta_n$
 on the node number $n$ in the Kuramoto model \eqref{eqn:dex1} with \eqref{eqn:ex1wc}
 for $a=1$ and $T=100$. }
\label{fig:5.1j'}
\end{center}
\end{figure}

Figure~\ref{fig:5.1j'} shows the dependence
 of the approximate convergence error $\Delta_n$ given by \eqref{eqn:error}
 on the node number $n$.
Here the initial values $u_i^n(0)$, $i\in[n]$, were independently randomly chosen
 according to the uniform distribution on $[-\pi,\pi]$ for each $n$ as in Fig.~\ref{fig:5.1g'}..
We observe that the error $\Delta_n$ decreases  as $n$ increases
 even though different initial conditions were taken.

\subsection{Determination of network graphs for prescribed desired solutions}

We next discuss a problem for determination of a network graph
 such that the Kuramoto model \eqref{eqn:dex1} exhibits desired synchronized motions
 approximately prescribed by
\begin{equation}
u_i^n(t)=\Omega_0 t+U_0(i/n)+\theta,\quad
\theta\in\Sset^1,\quad
i\in[n],
\label{eqn:solcex12}
\end{equation}
for $n>0$ sufficiently large,
 where $U_0(x)$ is a given function and $\theta\in\Sset^1$ is a constant,
 when the natural frequency $\omega_i^n$ of the $i$th node
 is given by a measurable function $\omega(x)$ through \eqref{eqn:omega} for $i\in[n]$ as in the above.
The synchronized motion is precisely represented by such a formula as \eqref{eqn:soldex1}.
We also assume the following:
\begin{itemize}
\setlength{\leftskip}{-1.6em}
\item[(i)]
There exist two sets $I_\pm\subset I$ with nonzero measures such that
\[
\mbox{$\sin U_0(x)>0$ for $x\in I_+$ and $\sin U_0(x)<0$ for $x\in I_-$;}
\]
\item[(ii)]
$(\omega(x)-\Omega_0)\sin U_0(x)\ge 0$;
\item[(iii)]
$\omega(x)-\Omega_0=0$ if $\sin U_0(x)=0$.
\end{itemize}
From the analysis of Section~3.1 we see that if condition~\eqref{con:2} holds,
 then the family of solutions approximately given by \eqref{eqn:solcex12} is asymptotically stable
 for $n>0$ sufficiently large.

Using \eqref{eqn:solcex1}, we obtain
\begin{equation}
\begin{split}
&
H_1(x)=
\begin{cases}
\displaystyle
\frac{\omega(x)-\Omega_0}{C_0\sin U_0(x)} & \mbox{if $\sin U_0(x)\neq 0$;}\\
0 & \mbox{otherwise},
\end{cases}\\
&
H_2(y)=
\begin{cases}
H_{2+} & \mbox{if $y\in I_+$;}\\
H_{2-} & \mbox{if $y\in I_-$;}\\
0 & \mbox{if $y\notin I_+\cup I_-$}
\end{cases}
\end{split}
\label{eqn:hcex12}
\end{equation}
as a solution to our problem, where
\begin{align*}
&
H_{2+}=-\int_{I-}\sin U_0(y)dy>0,\quad
H_{2-}=\int_{I+}\sin U_0(y)dy>0,\\
&
C_0=H_{2-}\int_{I_-}\cos U_0(y)dy+H_{2+}\int_{I_+}\cos U_0(y)dy.
\end{align*}
Note that the relation
\begin{align*}
\int_I H_2(y)\sin U_0(y)dy=0,
\end{align*}
which guarantees \eqref{eqn:rcex1}, holds.
Thus, depending on whether the network graph $G_n$ is deterministic dense,
random dense or random sparse,
we can obtain the weight matrix $W(G_n)$ from $W(x,y)=H_1(x)H_2(y)$
through \eqref{eqn:ddd}, \eqref{eqn:ddr} or \eqref{eqn:sdr}
 such that the Kuramoto model \eqref{eqn:dex1}
exhibits the desired motion \eqref{eqn:solcex12},
 where for the random dense network,
 the function $D(u)$ and the probability $\Pset(j\to i)$ are replaced
 by $w_{n0}D(u)$ and
\[
\Pset(j\to i)=\frac{1}{w_{n0}}\langle W(x,y)\rangle_{ij}\in[0,1],
\]
respectively, if $\sup_{(x,y)\in I}H_1(x)H_2(y)>1$,
where $w_{n0}>1$ is a normalized constant.

As an example, we consider the case in which
 the desired synchronized motion is approximately given by
\eqref{eqn:solcex12} with
\begin{equation}
\Omega_0=0,\quad
U_0(x)=
\begin{cases}
-\frac{1}{6}\pi & \mbox{for $x\in[0,0.5]$;}\\
\frac{1}{6}\pi & \mbox{for $x\in(0.5,1]$}
\end{cases}
\label{eqn:ex12a}
\end{equation}
when the natural frequencies are given by \eqref{eqn:omega} with
\[
\omega(x)=
\begin{cases}
\tfrac{1}{2}(x-1) & \mbox{for $x\in[0,0.5]$;}\\
\tfrac{1}{2}x & \mbox{for $x\in(0.5,1]$}.
\end{cases}
\]
The synchronized solution $u=U_0(x)+\theta$ to \eqref{eqn:cex1} with \eqref{eqn:ex12a}
 consists of two clusters.
We choose $I_-=[0.2,0.3]$ and $I_+=[0.7,0.8]$, so that
\[
H_{2\pm}=0.05,\quad
C_0=0.01\times\frac{\sqrt{3}}{2},\quad
H_1(x)=
\begin{cases}
\displaystyle
-C_0^{-1}(x-1) & \mbox{$x\in[0,0.5]$;}\\
C_0^{-1}x & \mbox{$x\in(0.5,1]$.}
\end{cases}
\]
Hence,
\[
W(x,y)=
\begin{cases}
\displaystyle
-\frac{10(x-1)}{\sqrt{3}} & \mbox{if $x\in[0,0.5]$ and $y\in[0.2,0.3]\cup[0.7,0.8]$};\\
\displaystyle
\frac{10x}{\sqrt{3}} & \mbox{if $x\in(0.5,1]$ and $y\in[0.2,0.3]\cup[0.7,0.8]$};\\
0 & \mbox{otherwise.}
\end{cases}
\]

\begin{figure}[t]
\begin{center}
\includegraphics[scale=0.26]{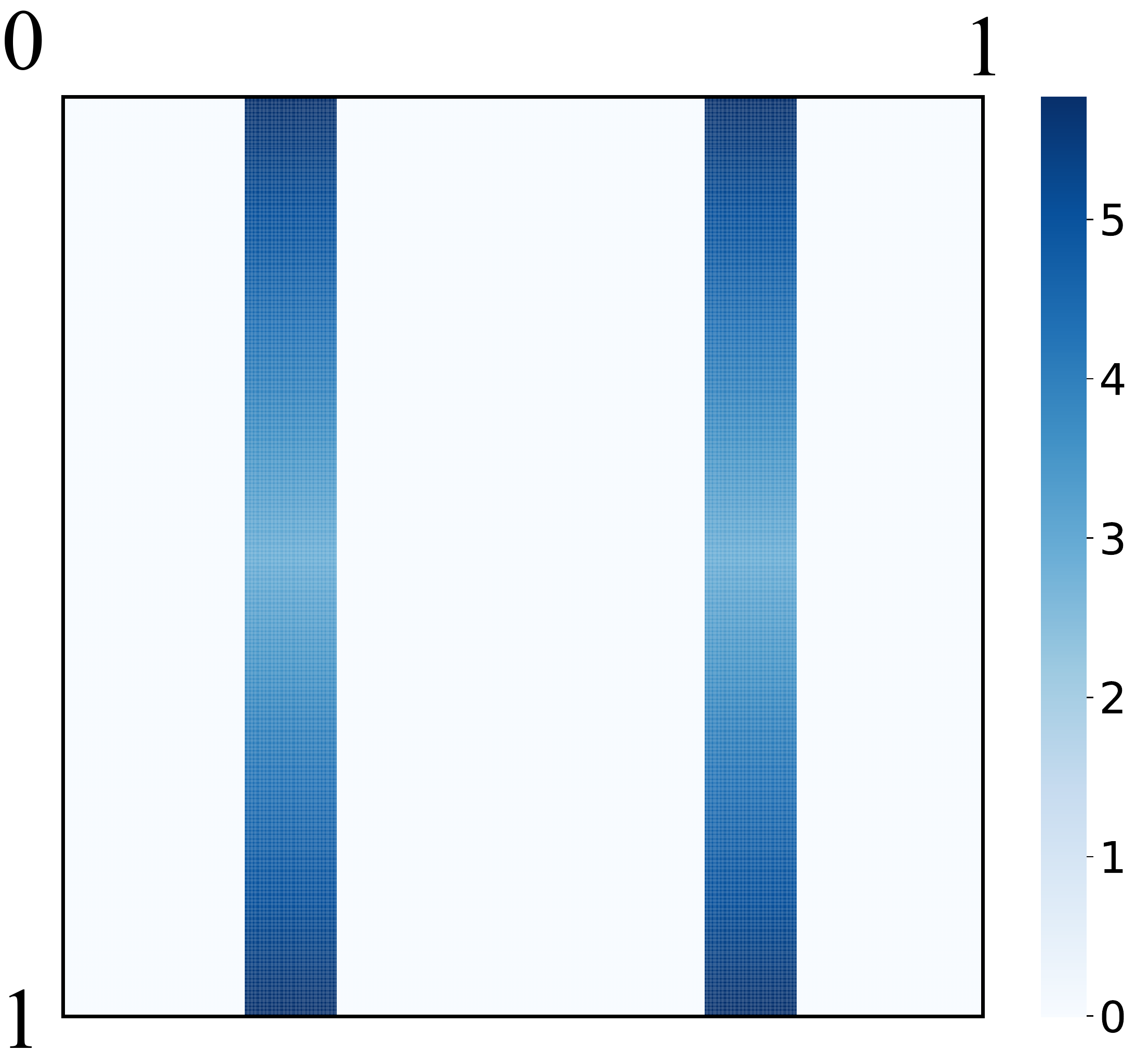}
\caption{Pixel picture of a sampled weight matrix
of the deterministic directed dense graph $G_n$ with $n=1000$.
The shade of the corresponding pixel represents the strength of $w_{ij}^n$.
Especially, its color is white if $w_{ij}^n=0$.
See also the caption of Fig.~\ref{fig:5.1e}.}
\label{fig:5.1k}
\end{center}
\end{figure}

Henceforth we assume that
the network graph $G_n$ is deterministic directed dense.
The other network graphs can be treated similarly, as in Sections~3.2.2 and 3.2.3.
Figure~\ref{fig:5.1k} represents the weight matrix of $G_n$ with $n=1000$.
Noting that
\[
\frac{1}{\cos U_i}-2\cos U_i=1,\quad
i\in[n],
\]
we see that condition~\eqref{con:2d'} holds as well as \eqref{con:2d},
so that by our general arguments to the Kuramoto model \eqref{eqn:dex1},
the family of desired synchronized solutions is asymptotically stable.
We carried out numerical simulations
for the Kuramoto model \eqref{eqn:dex1}
with the weight matrix displayed in Fig.~\ref{fig:5.1k}
for $n=1000$, $\alpha_n=1$ and $a=0.5$.
The initial values $u_i^n(0)$, $i\in[n]$, were independently randomly chosen on $[-\pi,\pi]$.

\begin{figure}[t]
\begin{center}
\includegraphics[scale=0.3]{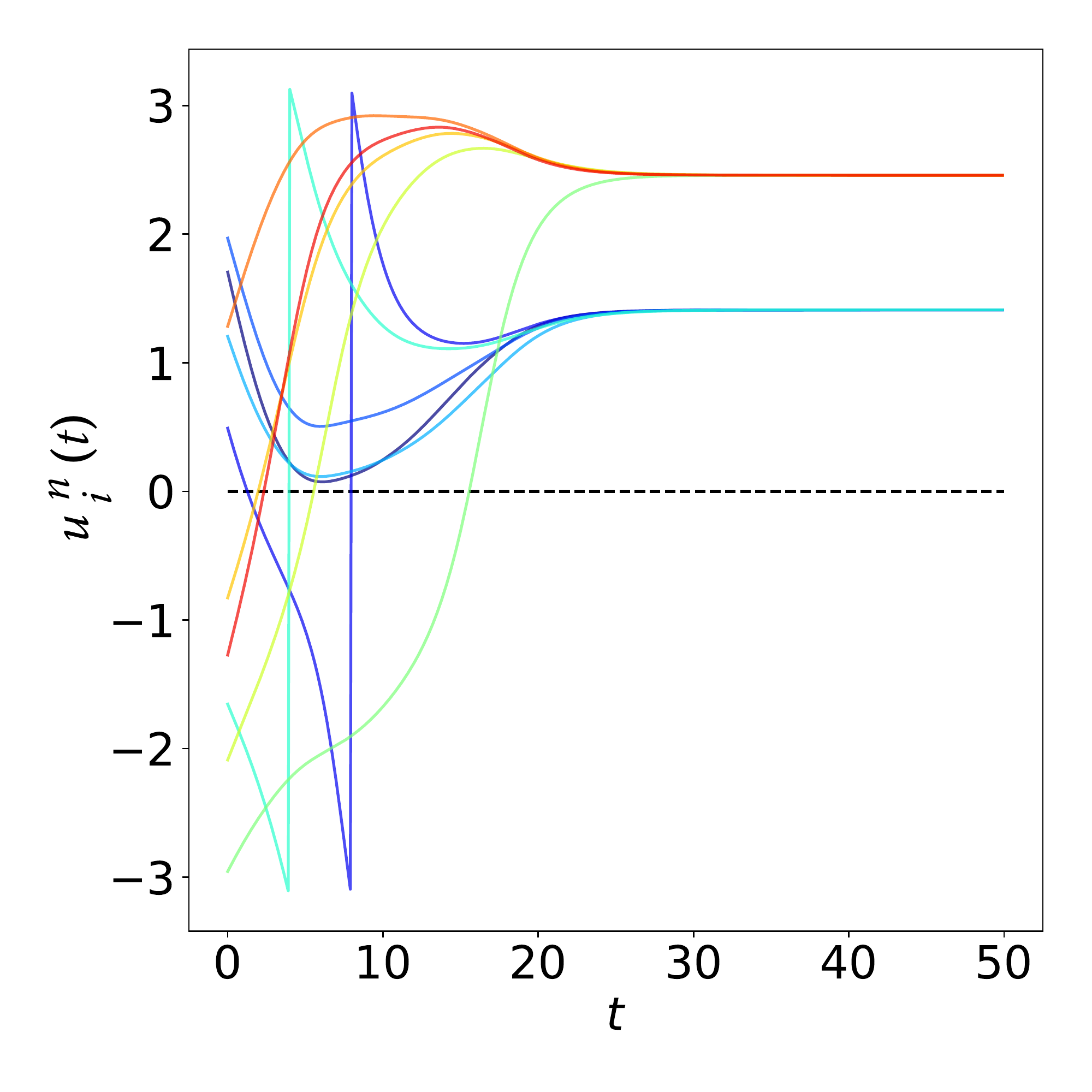}
\caption{Numerical simulation result
of the Kuramoto model \eqref{eqn:dex1}
with the weight matrix displayed in Fig.~\ref{fig:5.1k} for $n=1000$ and $a=0.5$.
See also the caption of Fig.~\ref{fig:5.1c}.}
\label{fig:5.1l}
\end{center}
\end{figure}

\begin{figure}[t]
\begin{center}
\includegraphics[scale=0.3]{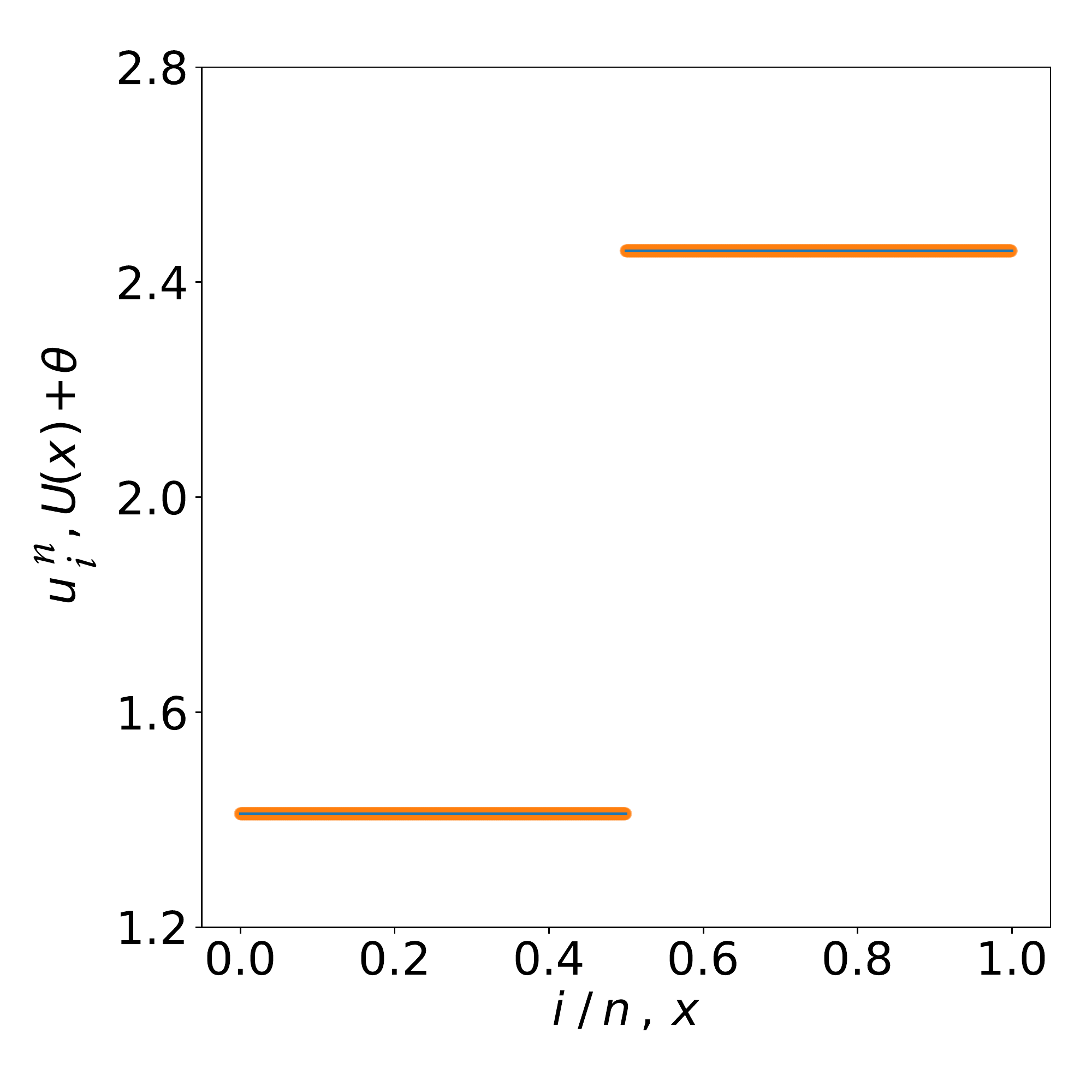}
\caption{The snapshot of the phases $u_i^n(50)$, $i\in[n]$,
in the Kuramoto model \eqref{eqn:dex1}
 with the weight matrix displayed in Fig.~\ref{fig:5.1k} for $n=1000$ and $a=0.5$.
Here $\theta=1.934663\ldots$ is estimated.}
\label{fig:5.1m}
\end{center}
\end{figure}

Figure~\ref{fig:5.1l} shows the time-history of every 100th node.
In Fig.~\ref{fig:5.1m} the snapshot of the phases $u_i^n(50)$, $i\in[n]$, is plotted.
We observe that the response rapidly converges to the desired synchronized state
\[
u_i^n=
\begin{cases}
-\frac{1}{6}\pi+\theta & \mbox{for $i\le 500$;}\\
\frac{1}{6}\pi+\theta & \mbox{for $i>500$}
\end{cases}
\]
with
\[
\theta\approx\frac{1}{n}\sum_{i=1}^n u_i^n(50)=1.934663\ldots.
\]


\section{Controlled Kuramoto model with natural frequencies}

We next consider a modified Kuramoto model with multiple natural frequencies,
\begin{equation}
\frac{d}{dt} u_i^n (t) = \omega_i^n+b\sin(\bar{U}-u_i^n(t))
+\frac{1}{n} \sum^{n}_{j=1}\sin \left(u_j^n(t) - u_i^n(t) \right),\quad i \in [n],
\label{eqn:dex2}
\end{equation}
and its continuum limit
\begin{equation}
\frac{\partial}{\partial t}u(t,x) = \omega(x) +b\sin(\bar{U}-u(t,x))
+ \int_I \sin( u(t,y)-u(t,x) ) dy, \quad x \in I,
\label{eqn:cex2}
\end{equation}
which are, respectively, special cases of \eqref{eqn:dskw} and \eqref{eqn:cskw}
with $m=1$, $f(u,t)=b\sin(\bar{U}-u)$, $D(u)=\sin u$ and $W(x,y)\equiv1$,
where $\bar{U}\in\Sset^1$ is a constant.
Similar systems were numerically studied in \cite{SA15,SA16}.
Here the dependence of the graph on the index $k$ is dropped out, as in Section~3.
A non-constant function could be taken as the desired state $\bar{U}$
 (cf. Eq.~\eqref{eqn:tf})
 but the analytic treatment of the continuum limit \eqref{eqn:cex2} would be very difficult.

Let $b=0$.
From Section~3 we see that if there is not a constant $C$
 satisfying \eqref{con:c} with $H_\ell(x)\equiv 1$, $\ell=1,2$,
 then both the continuum limit \eqref{eqn:cex2}
 and Kuramoto model \eqref{eqn:dex2} with $n>0$ sufficiently large
 do not have synchronized solutions of the forms \eqref{eqn:solcex1} and \eqref{eqn:soldex1}, respectively.
For example, when $\omega(x)=a(x-\frac{1}{2})$ as in Section~3.2,
 if $a>\frac{1}{2}\pi$, then such a constant $C$ does not exist.
So we try to choose an adequate value of $b$
so that the modified Kuramoto model \eqref{eqn:dex2}
exhibits a synchronized motion like \eqref{eqn:soldex1}.
In this case
 the second term `$b\sin(\bar{U}-u_i^n(t))$' is regarded
 as a nonlinear feedback control in \eqref{eqn:dex2}
 such that the coupled oscillator network \eqref{eqn:dex2} is desired
 to exhibit the completely synchronized state $u_i^n=\bar{U}$, $i\in[n]$.

We assume that the continuum limit \eqref{eqn:cex2} with $b\neq 0$
has a synchronized solution of the form \eqref{eqn:solcex1} with $\theta=\bar{U}$.
An argument similar to that of Section~3 can apply
although Eq.~\eqref{con:c} is replaced with
\begin{equation}
C-b=\int_I\cos U(y)dy.
\label{eqn:ccex2}
\end{equation}
So we see that if there exists a constant $C$ such that Eq.~\eqref{eqn:ccex2} holds,
then the synchronized solution \eqref{eqn:solcex1} with $\theta=\bar{U}$ exists.
Moreover, the linear operator \eqref{eqn:lcex1} is replaced with
\[
\L\phi(x)=\int_I(\phi(y)-\phi(x))\cos(U(y)-U(x))dy-b\phi(x)\cos U(x),
\]
so that
\begin{align*}
&
\langle\L\phi,\phi\rangle\\
&
=-\tfrac{1}{2}\int_{I^2}(\phi(y)-\phi(x))^2\cos(U(y)-U(x))dxdy
-b\int\phi(x)^2\cos U(x)dx.
\end{align*}
Hence, if condition~\eqref{con:2'} holds with $p=1$,
then the synchronized solution \eqref{eqn:solcex1} with $\theta=\bar{U}$
is asymptotically stable.
Similarly, we see that if there exists a constant $C_D$ such that
\[
C_D-b=\frac{1}{n}\sum_{i=1}^n\cos U_i,
\]
then the synchronized solution \eqref{eqn:soldex1} with $\theta=\bar{U}$ exists.
Since the relations \eqref{eqn:limitex1} also hold,
 the synchronized solutions of \eqref{eqn:dex2}
 converges to \eqref{eqn:solcex1} of \eqref{eqn:cex2} as $n\to\infty$.
Moreover, we show that if condition~\eqref{con:2'} holds with $p=1$,
 then the solution \eqref{eqn:soldex1} is asymptotically stable for $n>0$ sufficiently large,
 as in Section~3.1.
 
As an example, we consider the case in
which $\omega(x)=a(x-\frac{1}{2})$ and $\bar{U}=\pi/6$.
We have $\Omega=0$ and compute \eqref{eqn:ccex2} as
\[
C-b=\frac{C}{a} \left(\arcsin\left(\frac{a}{2C} \right)
+\frac{a}{2C} \sqrt{1- \left( \frac{a}{2C} \right)^2 } \right)
\]
like \eqref{eqn:acex1'}, so that
\begin{equation}
a=\left(\arcsin\left(\frac{a}{2C} \right)
+\frac{a}{2C} \sqrt{1- \left( \frac{a}{2C} \right)^2 } \right)+2b\left(\frac{a}{2C}\right).
\label{eqn:ccex2'}
\end{equation}
We see that if and only if
\begin{equation}
b\ge\frac{4a-2-\pi}{4\sqrt{2}},
\label{con:12}
\end{equation}
then there exists a constant $C$ satisfying \eqref{eqn:ccex2'}
 and condition~\eqref{con:2'} with $p=1$ holds,
 so that the synchronized solution \eqref{eqn:solcex1a} with $\theta=\pi/6=0.5235987\ldots$
 exists in the continuum limit \eqref{eqn:cex2} and it is asymptotically stable,
 along with the solution \eqref{eqn:soldex1a} with $\theta=\pi/6$
 in the discrete model \eqref{eqn:dex2} for $n>0$ sufficiently large.

\begin{figure}[t]
\begin{center}
\includegraphics[scale=0.3]{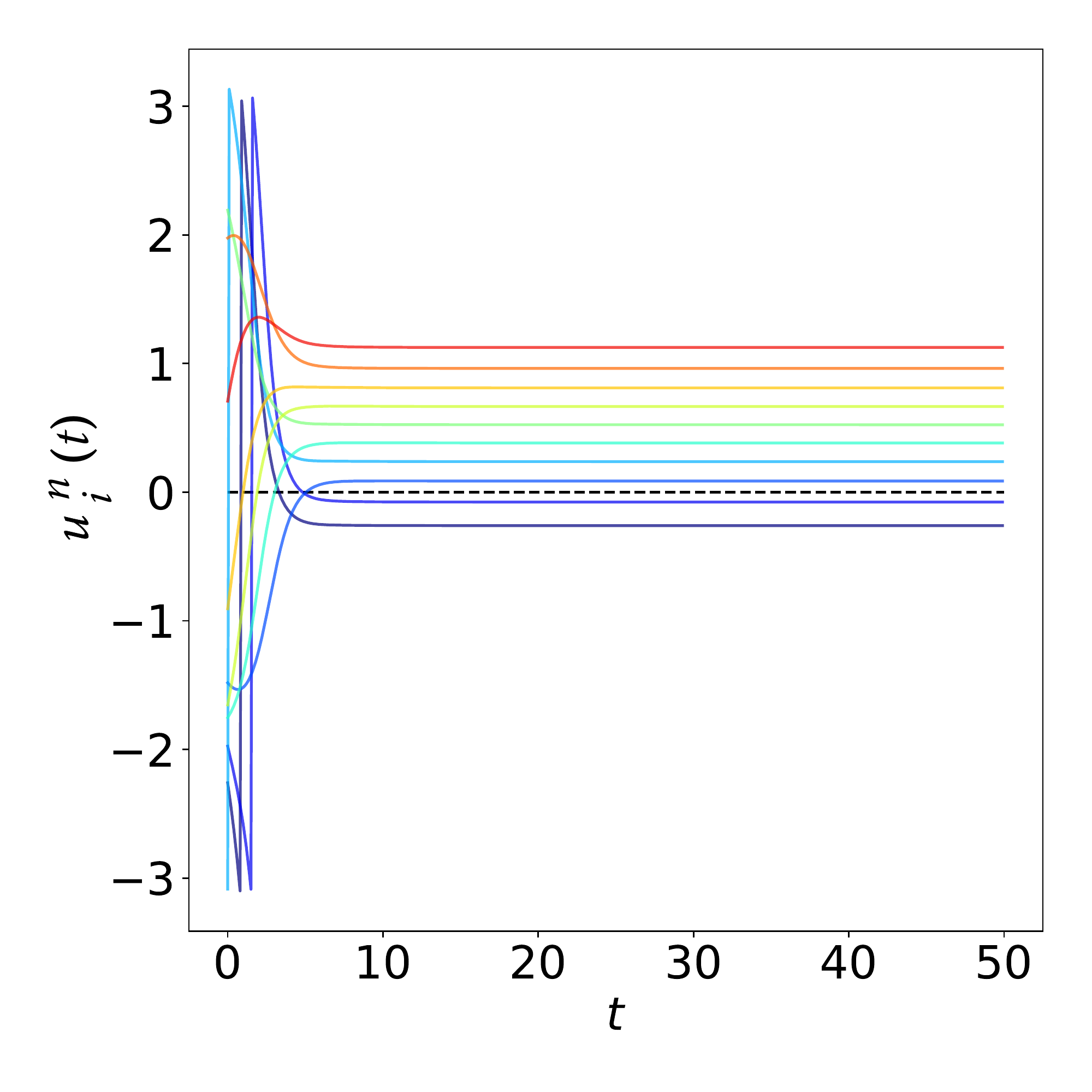}
\caption{Numerical simulation result
of the controlled Kuramoto model \eqref{eqn:dex2}
with $\omega(x)=a(x-\frac{1}{2})$ for $n=1000$, $a=2$ and $b=0.506$.
See also the caption of Fig.~\ref{fig:5.1c}.}
\label{fig:5.2a}
\end{center}
\end{figure}

\begin{figure}[t]
\begin{center}
\includegraphics[scale=0.3]{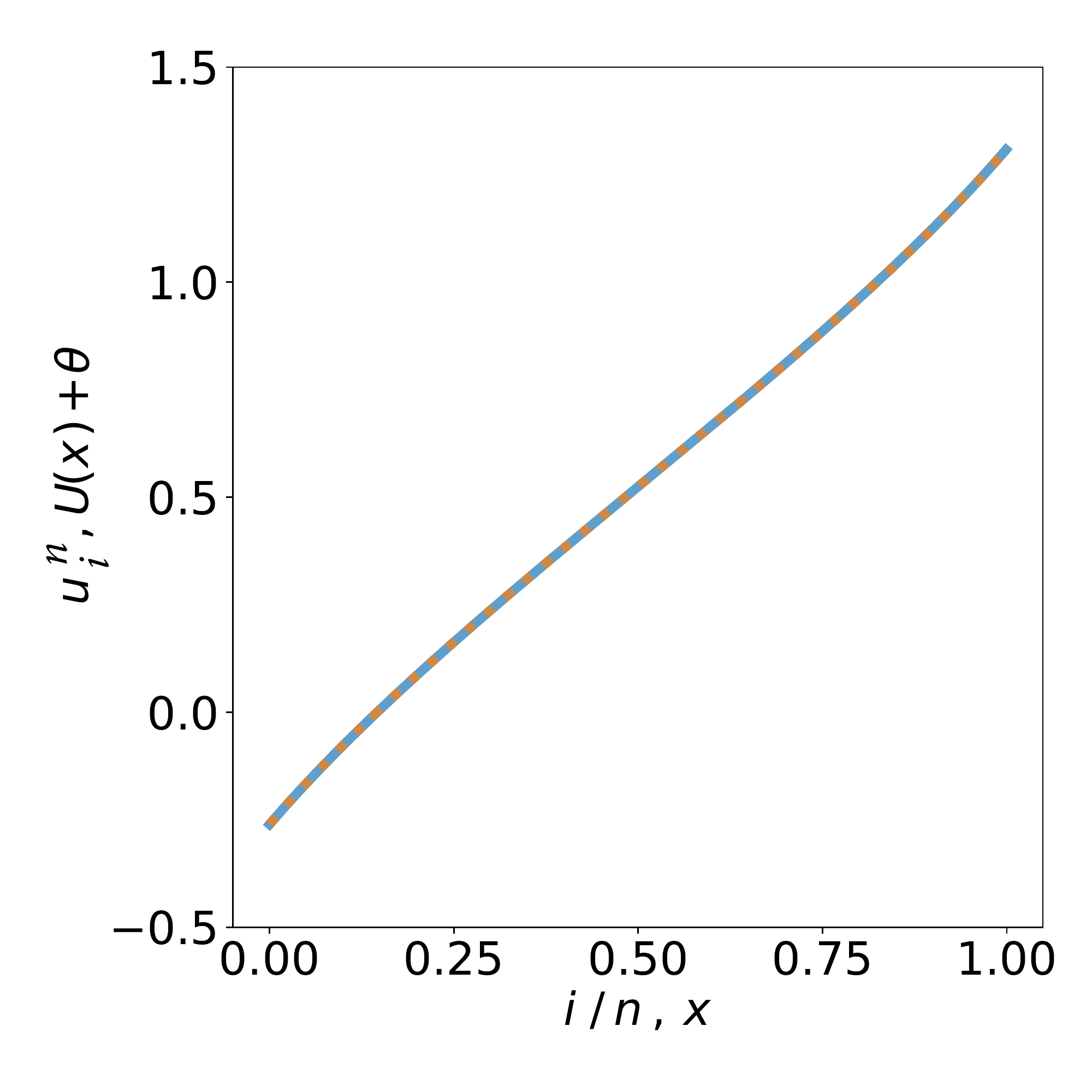}
\caption{Comparison between $u_i^n(50)$, $i\in[n]$,
to the controlled Kuramoto model \eqref{eqn:dex2} with $\omega(x)=a(x-\frac{1}{2})$
and the synchronized solution \eqref{eqn:solcex1a} to the continuum limit \eqref{eqn:cex2}
with $a/(2C)=0.70669\ldots$ and $\theta=0$ for $n=1000$, $a=2$ and $b=0.506$.
The former and latter are plotted as orange dotted and blue solid lines, respectively.}
\label{fig:5.2b}
\end{center}
\end{figure}

We carried out a numerical simulation
for the controlled Kuramoto model \eqref{eqn:dex2} with $n=1000$, $a=2$ and
\[
b=0.506>\frac{6-\pi}{4\sqrt{2}}=0.5052998\ldots,
\]
which yields $a/(2C)=0.70699\ldots$ by \eqref{eqn:ccex2'}.
The initial values $u_i^n(0)$, $i\in[n]$, were independently randomly chosen on $[-\pi,\pi]$.
Figure~\ref{fig:5.2a} shows the time-history of every 100th node.
We observe that the response rapidly converges to a synchronized state.
In Fig.~\ref{fig:5.2b} the response of \eqref{eqn:dex2} at $t=50$ for $n=1000$
is compared with the corresponding continuum limit synchronized solution \eqref{eqn:solcex1a}.
We see that both coincide almost completely, as predicted theoretically.


\section{Kuramoto model with no natural frequencies on two graphs}

We finally consider another modified Kuramoto model with no natural frequencies
depending on two graphs,
\begin{align}
\frac{d}{dt} u_i^n (t) &= \frac{1}{n}\sum^{n}_{j=1}
\sin \left( u_j^n(t) - u_i^n(t) \right) \notag\\
& \quad + \frac{K}{n \alpha_n} \sum^{n}_{j=1}
w_{ij} \sin 2 \left( u_j^n(t) - u_i^n(t) \right) ,\quad i \in [n],
\label{eqn:dex3}
\end{align}
and its continuum limit
\begin{align}
\frac{\partial}{\partial t}u(t,x) &=
\int_I \sin( u(t,y)-u(t,x) ) dy \notag\\
& \quad + K \int_I W(x,y) \sin 2( u(t,y)-u(t,x) ) dy , \quad x \in I,
\label{eqn:cex3}
\end{align}
where $K$ is a constant.
Here one of the graphs is a complete graph,
and the dependence of the other graph on the index $k$ is dropped out.
It is well known \cite{WSG06} that
the Kuramoto model with no natural frequencies on a single complete graph, i.e.,
Eq.~\eqref{eqn:dex3} with $K=0$,
exhibits a complete synchronized state $u_i^n(t)= q$, $i\in[n]$, where $q\in\Sset^1$ is a constant.
So we look for the value of $K\neq 0$
such that the complete synchronized state becomes unstable.
As the second graph, we choose a $\kappa$-nearest neighbor graph given by
\begin{equation*}
w_{ij} =
\begin{cases}
1 & \mbox{if $|i-j|\leq\kappa n$ or $|i-j|\geq (1-\kappa)n$;}\\
0 & \mbox{otherwise}
\end{cases}
\end{equation*}
with $0 <\kappa\leq 1/2$, so that
\begin{equation*}
W(x,y)=
\begin{cases}
1 & \mbox{if $|x-y|\leq\kappa$ or $|x-y|\geq 1-\kappa$;}\\
0 & \mbox{otherwise.}
\end{cases}
\end{equation*}
In particular, the graph is deterministic, undirected and dense,
and $\alpha_n=1$ in \eqref{eqn:dex3}.

We first discuss the linear stability
 of the complete synchronized solution $u(t,x)=q$ in \eqref{eqn:cex3},
which corresponds to the solution $u_i^n(t)=q$ in \eqref{eqn:dex3}.
The associated linear operator $\L:L^2(I)\to L^2(I)$ is given by
\begin{align}
\L \phi(x) 
&= \int_I (1 + 2K W(x,y)) \phi(y) dy - (1 + 4K\kappa) \phi(x).
\label{eqn:1-1L}
\end{align}
We see that $\phi(x)=1$ is an eigenfunction of $\L$ associated with the zero eigenvalue.
Moreover,
\begin{equation}
\phi(x)=\sin 2 \pi\ell x,\quad
\cos 2 \pi\ell x
\label{eqn:efex3}
\end{equation}
are eigenfunctions associated with the eigenvalue
\begin{equation}
\lambda = \frac{2K}{\pi\ell} \sin 2 \pi\ell\kappa-(1+4K\kappa)
\label{eqn:evex3}
\end{equation}
for each $\ell\in\Nset$ since
\begin{align*}
&
\int_I W(x,y) \sin 2 \pi\ell y\, dy
= \frac{1}{\pi\ell} \sin 2 \pi\ell \kappa\,\sin 2 \pi\ell x,\\
&
\int_I W(x,y) \cos 2 \pi\ell y dy
= \frac{1}{\pi\ell} \sin 2 \pi \ell\kappa\,\cos 2 \pi\ell x.
\end{align*}
These eigenvalues are the only ones of $\L$
 since the Fourier expansion of any function in $L^2(I)$ converges a.e.
 by Carleson's theorem \cite{C66}.
We compute
\begin{align*}
\langle \L \phi, \phi \rangle
&= \int_{I^2} (1 + 2K W(x,y) ) \phi(x) \phi(y) dxdy
- \int_{I^2} (1 + 4K\kappa) \phi(x)^2 dx \\
&= - \frac{1}{2} \int_{I^2} (1+ 2KW(x,y)) (\phi(x)-\phi(y))^2 dxdy \\
& \quad + \frac{1}{2} \int_{I^2} (1 + 2K W(x,y)) (\phi(x)^2+\phi(y)^2) dxdy
- \int_{I^2} (1+4K\kappa) \phi(x)^2 dx \\
&= - \frac{1}{2} \int_{I^2} (1+ 2KW(x,y)) (\phi(x)-\phi(y))^2 dxdy \\
& \quad + \int_{I^2} (1 + 2K W(x,y)) \phi(x)^2dxdy
- \int_{I^2} (1+4K\kappa) \phi(x)^2dx \\
&= - \frac{1}{2} \int_{I^2} (1+ 2KW(x,y)) (\phi(x)-\phi(y))^2 dxdy
\end{align*}
since
\[
\int_{I^2}W(x,y)\phi(x)^2dx
=2\kappa\int_I\phi(x)^2dx=2\kappa\int_{I^2}\phi(x)^2dxdy.
\]
Hence, if $K\geq 0$, then the complete synchronized solution is linearly stable.

\begin{figure}[t]
\begin{center}
\includegraphics[scale=0.6]{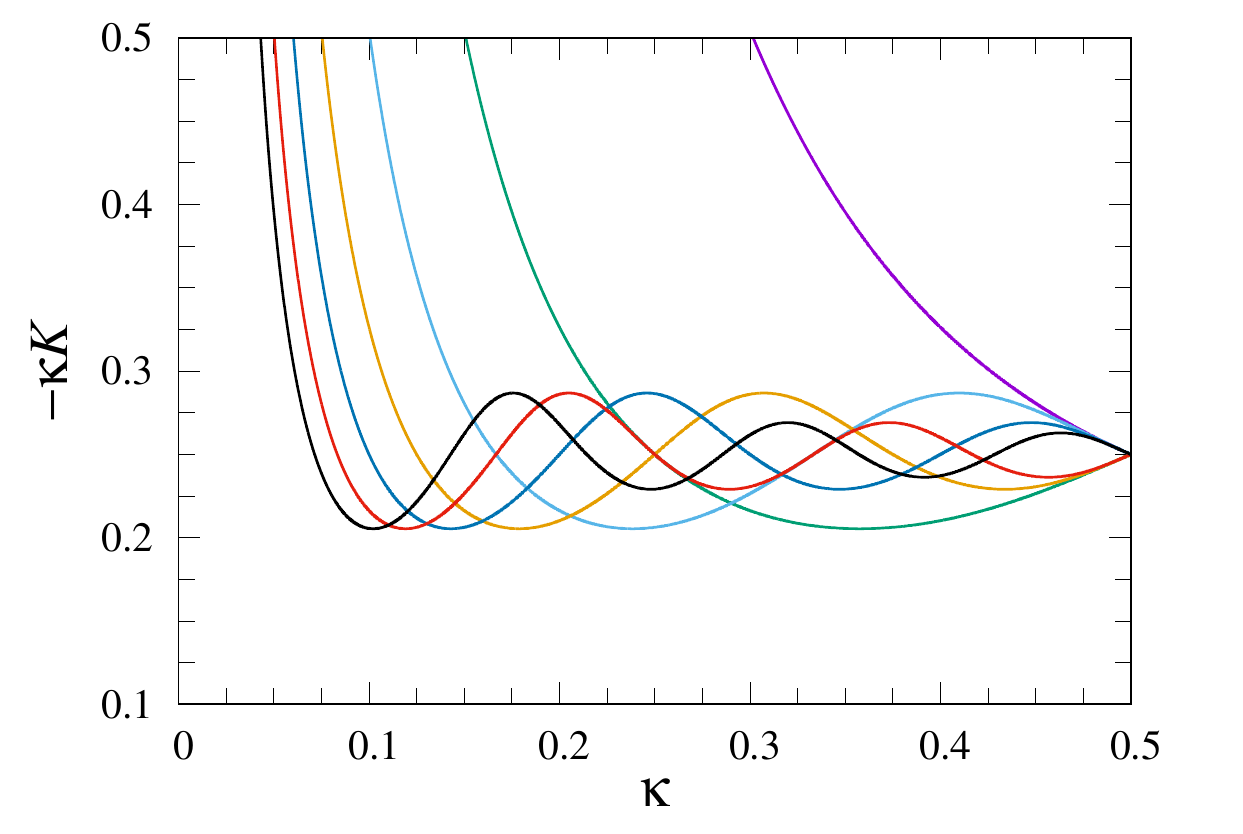}
\caption{Boundaries of the unstable regions given by \eqref{eqn:br} with $\ell=1$-$7$
for the complete synchronized solutions to \eqref{eqn:cex3}.
The boundaries of $\ell=1,2,3,4,5,6$ and $7$ are plotted
as purple, green, light blue, orange, blue, red and black lines. }
\label{fig:5.3a}
\end{center}
\end{figure}

Assume that $K<0$.
From \eqref{eqn:evex3} we see that if
\begin{equation}
\frac{2K}{\pi\ell} \sin 2 \pi\ell\kappa-(1+4K\kappa)>0\quad\mbox{i.e.,}\quad
-\kappa K>\frac{\pi\ell\kappa}{2(2\pi\ell\kappa-\sin 2\pi\ell\kappa)}
\label{eqn:br}
\end{equation}
for some $\ell\in\Nset$,
then the complete synchronized solution $u(t,x)=q$ is unstable for any constant $q$.
We remark that Eq.~\eqref{eqn:br} only gives a sufficient condition
 for $u(t,x)=q$ to be unstable.
We plot the boundaries of the unstable regions for the complete synchronized solutions
 given by \eqref{eqn:br} with $\ell=1$-$7$ in Fig.~\ref{fig:5.3a}.
By Theorem~\ref{thm:main4} (see also Remark~\ref{rmk:2d}(i)),
 the corresponding complete synchronized solutions to \eqref{eqn:dex3}
 for $n>0$ sufficiently large
 are also unstable in the region above any one of the curves in the figure.

\begin{figure}[t]
\begin{center}
\includegraphics[scale=0.3]{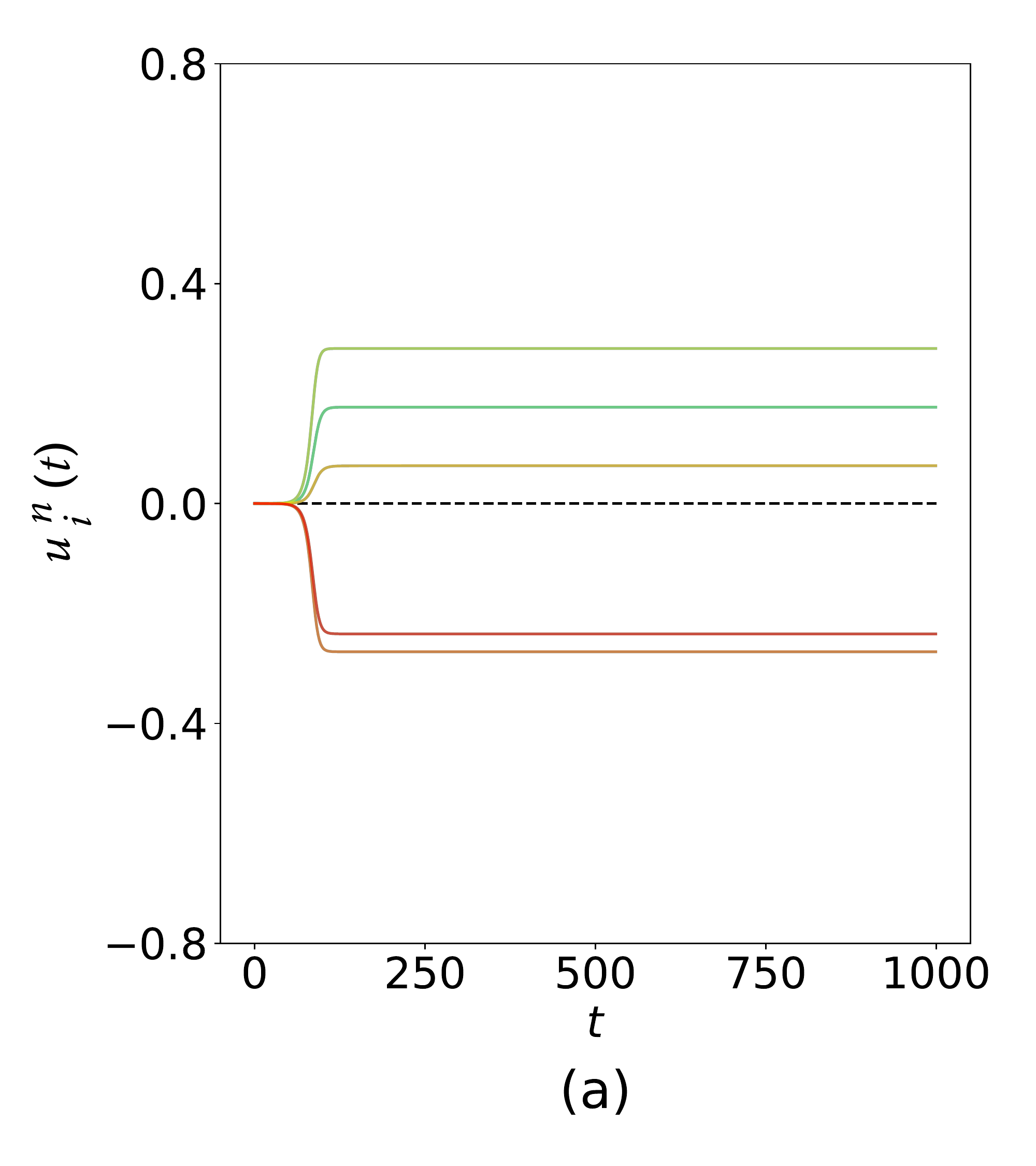}
\includegraphics[scale=0.3]{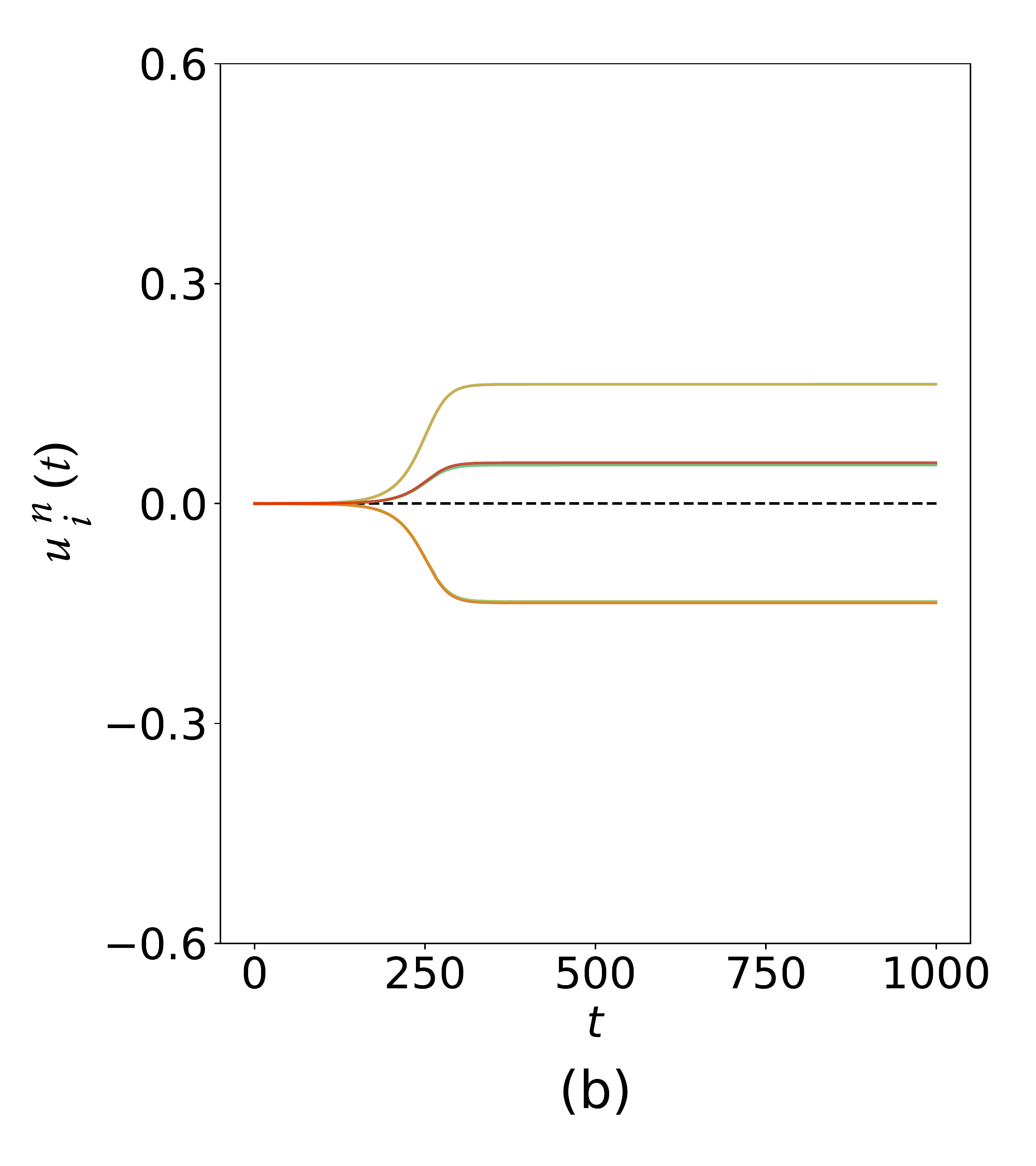}
\caption{Numerical simulation results of the modified Kuramoto model \eqref{eqn:dex3}
for $n=1000$:
(a) $\kappa=1/3$ and $K=-0.7$;
(b) $\kappa=1/8$ and $K=-1.7$.
See also the caption of Fig.~\ref{fig:5.1b}.}
\label{fig:5.3b}
\end{center}
\end{figure}

\begin{figure}[t]
\begin{center}
\includegraphics[scale=0.3]{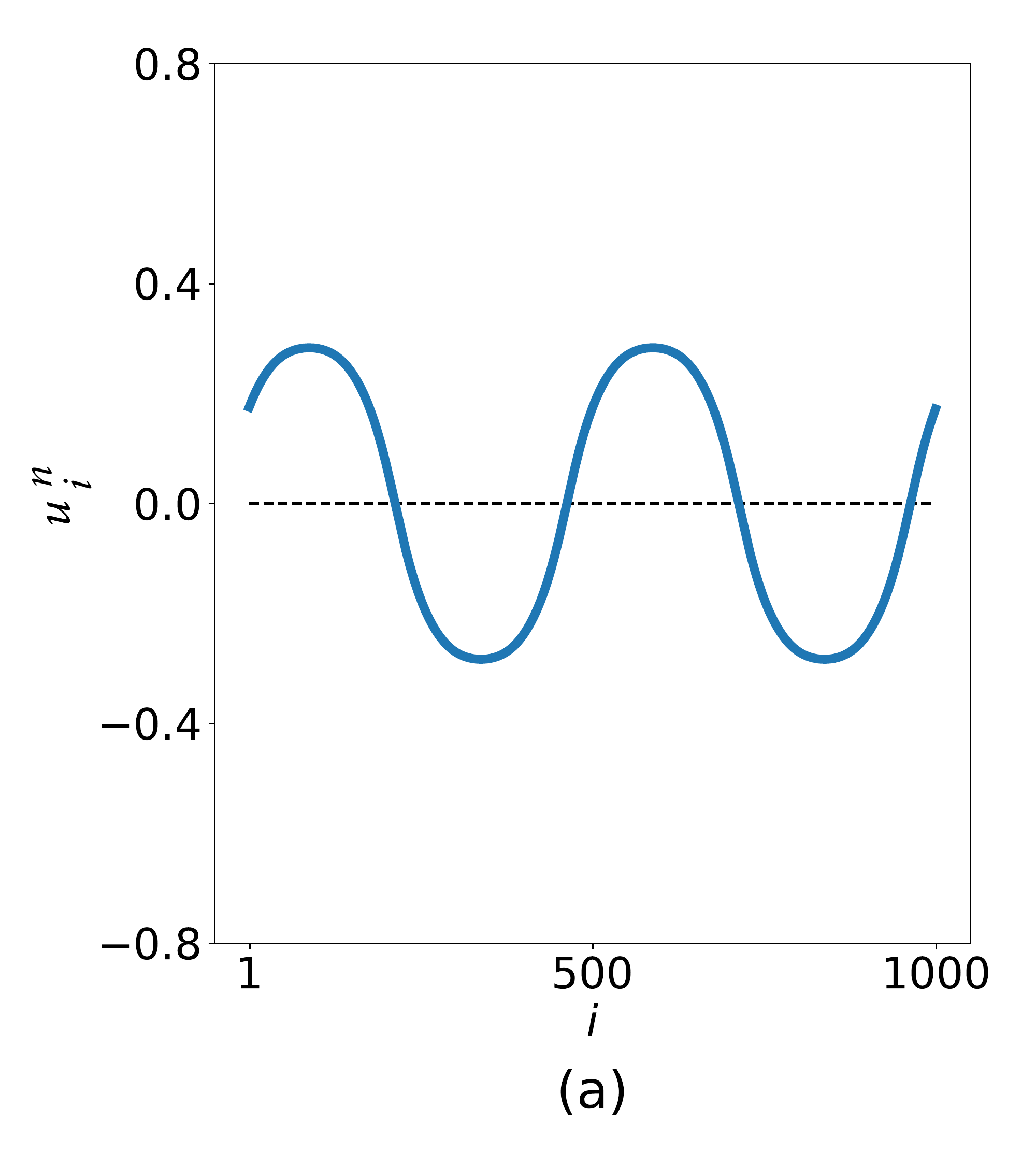}
\includegraphics[scale=0.3]{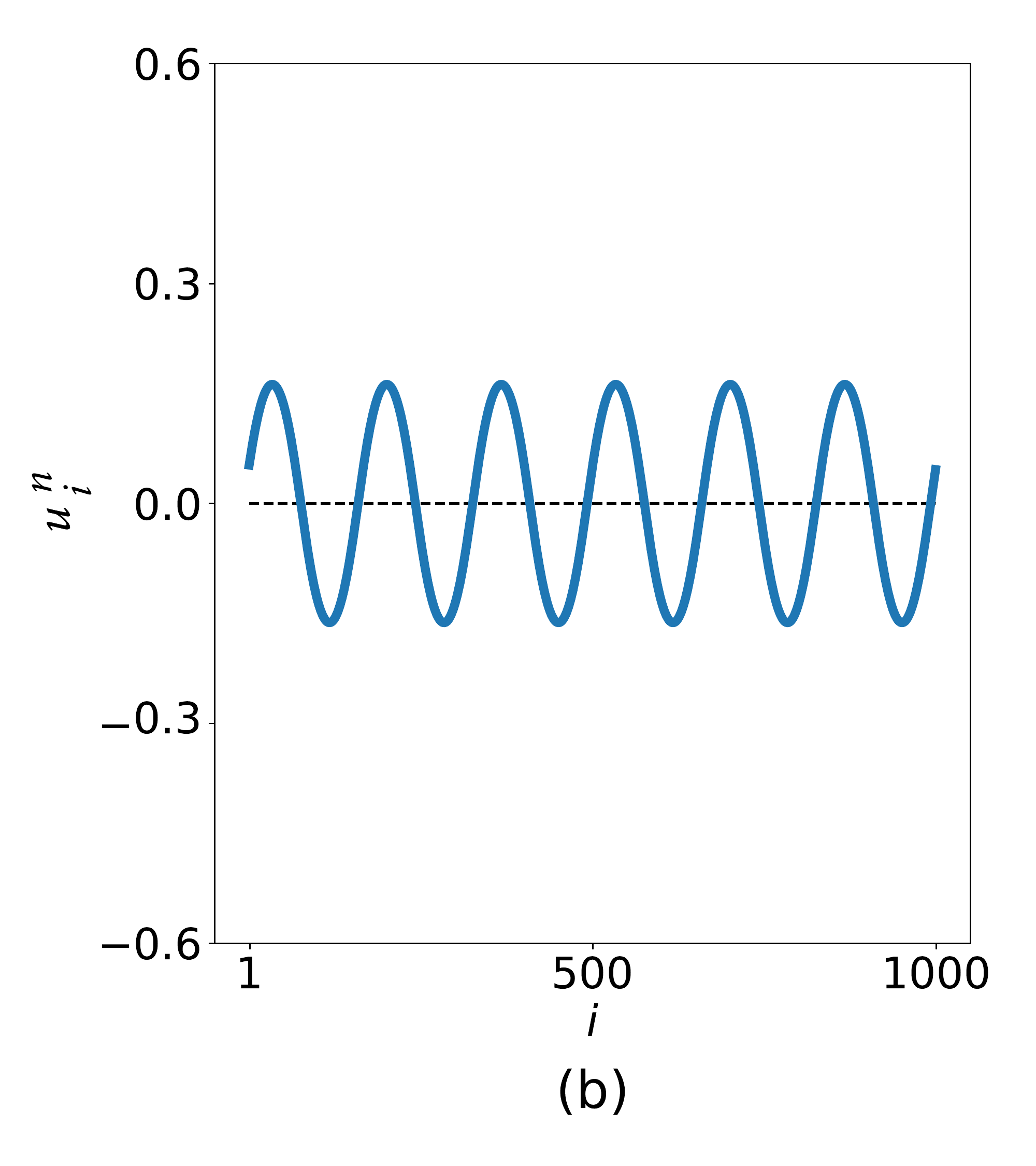}
\caption{Steady state of the modified Kuramoto model \eqref{eqn:dex3} for $n=1000$:
(a) $\kappa=1/3$ and $K=-0.7$; (b) $\kappa=1/8$ and $K=-1.7$.
Here $u_i^n(1000)$, $i\in[n]$, are plotted.}
\label{fig:5.3c}
\end{center}
\end{figure}

We carried out numerical simulations
for the modified Kuramoto model \eqref{eqn:dex3}
with $n=1000$ and $\kappa=1/3$ or $1/8$.
The weakest condition of \eqref{eqn:br} is $K<-0.621504\ldots$ with $\ell=2$
and $K<-1.64988\ldots$ with $\ell=6$ for $\kappa=1/3$ and $1/8$, respectively.
The initial values $u_i^n(0)$, $i\in[n]$,
were independently randomly chosen on $[-10^{-4},10^{-4}]$.
So if a complete synchronized state is asymptotically stable,
then the response converges to it.
Figures~\ref{fig:5.3b}(a) and (b) show the time-histories of every $100$th node
for $(\kappa,K)=(1/3,-0.7)$ and $(1/8,-1.7)$, respectively.
We observe that the response does not converge to a complete synchronized state
for both cases, as predicted theoretically.
In Figs.~\ref{fig:5.3c}(a) and (b),
$u_i^n(1000)$, $i\in[n]$, in the modified Kuramoto model \eqref{eqn:dex3} with $n=1000$,
 which may be regarded as the steady states,
are plotted for $(\kappa,K)=(1/3,-0.7)$ and $(1/8,-1.7)$, respectively.
A stable oscillatory state with two (resp. six) maxima and minima appears
like the eigenfunctions \eqref{eqn:efex3} with a positive eigenvalue
for the former (resp. latter) case
 although its existence and asymptotic stability is not obtained theoretically
 unlike the previous two examples.

\appendix

\renewcommand{\theequation}{\Alph{section}.\arabic{equation}}
\setcounter{equation}{0}


\section{Proof of Theorem~\ref{thm:main1}}

In this appendix, we extend arguments in the proof of Theorem~3.1 of \cite{KM17}
and prove Theorem~\ref{thm:main1}.
Let $L_f$ and $L_k$ be, respectively, the Lipschitz constants
of $f(\cdot,t)$ and $D_k(\cdot)$ for $k\in[m]$, and let
\[
L_D=\max_{k\in[m]}L_k.
\]
Let $\|\cdot\|$ stand for the norm in $L^2(I)$.

\begin{proof}[Proof of Theorem~$\ref{thm:main1}$]
Let
\begin{equation}
T=\left(2(L_f+mL_D(C_2+\|W\|_{L^2(I^2)}))\right)^{-1}.
\label{eqn:T}
\end{equation}
Let $\F=C(0,T;L^2(I))$ and define
\[
K(\mathbf{u})(t)
=g+\sum_{k=1}^m\int_0^t\biggl(f(u(s,\cdot),s)+\int_I W(\cdot,y)D_k(u(s,y)-u(s,\cdot))dy\biggr)ds
\]
for $\mathbf{u}\in\F$.
Obviously, if $\mathbf{u}\in\F$, then $K(\mathbf{u})\in\F$.
We easily see that a fixed point of the map $K:\F\to\F$, i.e.,
\[
\mathbf{u}=K(\mathbf{u}),
\]
gives a solution to the IVP of \eqref{eqn:csk} with \eqref{eqn:icc}.
In the following we show that $K$ is a contraction on $\F$,
 so that by the contraction map theorem
 (see, e.g., Theorem~2.2 in Chapter~2 of \cite{CH82})
 it has a unique fixed point in $\F$
 and the IVP of \eqref{eqn:csk} with \eqref{eqn:icc} has a unique solution.

For any $\mathbf{u}(t),\mathbf{v}(t)\in\F$, by the Lipschitz continuity of $f$ and $D_k$,
we compute
\begin{align*}
&\|K(\mathbf{u})-K(\mathbf{v})\|
 =\max_{t\in[0,T]}\|K(\mathbf{u}(t))-K(\mathbf{v}(t))\|\\
&
\le\max_{t\in[0,T]}\int_0^t\biggl\|f(u(s,\cdot),s)-f(v(s,\cdot),s)\\
&
\quad +\sum_{k=1}^m\int_I W(\cdot,y)(D_k(u(s,y)-u(s,\cdot))-D_k(v(s,y)-v(s,\cdot))dy\biggr\|ds\\
&
\le T\max_{t\in[0,T]}\biggl\|L_f|u(t,\cdot)-v(t,\cdot)|\\
&
\quad +L_D\sum_{k=1}^m\int_I W(\cdot,y)|u(t,y)-u(t,\cdot)-v(t,y)+v(t,\cdot)|dy\biggr\|\\
&
\le T\max_{t\in[0,T]}\biggl(L_f\|\mathbf{u}(t)-\mathbf{v}(t)\|
+L_D\sum_{k=1}^m\biggl(\biggl\|\int_I W(\cdot,y)|u(t,\cdot)-v(t,\cdot)|dy\biggr\|\\
&
\quad
+\biggl\|\int_I W(\cdot,y)|u(t,y)-v(t,y)|dy\biggr\|\biggr).
\end{align*}
Noting that by \eqref{eqn:assumpy}
\begin{align*}
\biggl\|\int_I W(\cdot,y)|u(t,\cdot)-v(t,\cdot)|dy\biggr\|
\le C_2\|\mathbf{u}(t)-\mathbf{v}(t)\|
\end{align*}
and by Schwarz' inequality
\[
\biggl\|\int_I W(\cdot,y)|u(t,y)-v(t,y)|dy\biggr\|
\le\|W\|_{L^2(I^2)}\,\|\mathbf{u}(t)-\mathbf{v}(t)\|,
\]
we obtain
\begin{align*}
&
\|K(\mathbf{u})-K(\mathbf{v})\|\\
&\le T(L_f+mL_D(C_2+\|W\|_{L^2(I^2)}))\|\mathbf{u}(t)-\mathbf{v}(t)\|
=\frac{1}{2}\|\mathbf{u}(t)-\mathbf{v}(t)\|.
\end{align*}
Here the last equality in the above equation holds due to \eqref{eqn:T}.
Thus, the IVP of \eqref{eqn:csk} with \eqref{eqn:icc}
has a unique solution $\bar{\mathbf{u}}(t)$ on $[0,T]$.
Using the standard arguments given in the last paragraph
 in the proof of Theorem~3.1 of \cite{KM17},
 we can extend the solution $\bar{\mathbf{u}}(t)$ to $(-\infty,\infty)$
 and show that it is continuously differentiable.
Actually, for example, the right-hand-side of \eqref{eqn:csk} is continuous
 if $\bar{\mathbf{u}}\in C(\Rset;L^2(I))$.
Moreover, since $K:\F\to\F$ is a uniform contraction and depends on $g$ continuously,
 the unique solution is a continuous function of $g$.
Thus, we complete the proof.
\end{proof}


\section{Proof of Theorem~\ref{thm:main2}}

In this appendix, we extend arguments in the proof of Theorem~3.1 of \cite{M19}
and prove Theorem~\ref{thm:main2}.
Henceforth we assume that the hypotheses of Theorem~\ref{thm:main2} hold.
So there exists a positive constant $C_D$ such that
\begin{equation*}
D_k(u)<C_D\quad\mbox{for $u\in\Rset$ and $k\in[m]$.}
\end{equation*}

Changing the order of the graphs if necessary,
we assume that for some $m_\d \in \{0\} \cup [m]$,
$G_{kn}$ is a deterministic or random graph for $n\in\Nset$
depending on whether $k \leq m_\d$ or $k>m_\d$.
All of $G_{kn}$, $k\in[m]$, are random if $m_\d =0$,
and deterministic if $m_\d =m$.
We average the coefficients appearing in the second term
in the right-hand-side of \eqref{eqn:dsk} for $k\in[m]\setminus[m_\d]$ as
\[
\frac{1}{n\alpha_{kn}}\Eset[w_{ij}^{kn}]=\frac{1}{n}\bar{W}_{ij}^{kn},\quad
\bar{W}_{ij}^{kn}=\langle\tilde{W}\rangle_{ij}^n,
\]
where $\alpha_{kn}=1$ and $\tilde{W}=W$ when $G_{kn}$ is dense,
 and consider the averaged model
\begin{align}
\frac{d}{dt} v_i^n(t) &= f(v_i^n, t) +
\sum^{m_\d}_{k=1} \frac{1}{n} \sum^{n}_{j=1} w_{ij}^{kn} D_k(v_j^n(t) - v_i^n(t))\notag\\
& \quad + \sum^{m}_{k=m_\d + 1} \frac{1}{n} \sum^{n}_{j=1} \bar{W}_{ij}^{kn} D_k(v_j^n(t) - v_i^n(t)),
\quad i \in [n].
\label{eqn:ask}
\end{align}
Let $u_n(t)$ and $v_n(t)$ denote the solutions to the IVPs
of \eqref{eqn:dsk} and \eqref{eqn:ask} with \eqref{eqn:icd} and
\begin{equation}
v_i^n(0) = u_{i0}^n,
\label{eqn:ica}
\end{equation}
respectively.
We adopt the discrete $L^2$-norm
\begin{equation*}
\|u_n\|_{2,n}
:=\biggl(\frac{1}{n}\sum_{i=1}^n (u_i^n)^2\biggr)^{1/2}.
\end{equation*}
We obtain the following estimate on the difference between the solutions $u_n(t)$ and $v_n(t)$.

\begin{lem}
\label{lem:b1}
Suppose that $\gamma_{\mathrm{max}}:=\max_{k \in [m]}\gamma_k\in(0,\frac{1}{2})$.
Then for any $T>0$ and $\delta\in (0,\tfrac{1}{2}-\gamma_{\rm max} )$
we have
\begin{equation}
\lim_{n \rightarrow \infty} n^{\frac{1}{2}-\gamma_{\rm max}-\delta}
\max_{t \in [0,T]} \|u_n(t) - v_n(t)\|_{2,n} = 0 \quad \mbox{a.s.}
\label{eqn:lemb}
\end{equation}
In particular,
\begin{equation*}
\lim_{n \rightarrow \infty}\max_{t \in [0,T]}\|u_n(t) - v_n(t)\|_{2,n}=0 \quad
\mbox{a.s.}
\end{equation*}
\end{lem}

\begin{proof}

Let $\psi_i^n := v_i^n - u_i^n$.
Subtracting \eqref{eqn:dsk} from \eqref{eqn:ask},
multiplying the resulting equation by $n^{-1}\psi_i^n$,
and summing over $i \in [n]$, we obtain
\begin{equation}
\frac{1}{2} \frac{d}{dt} \| \psi_n \|^{2}_{2,n}=I_1+I_2+I_3,
\label{eqn:lemb1}
\end{equation}
where
\begin{align*}
I_1=&\frac{1}{n} \sum^{n}_{i=1} ( f(v_i^n,t) - f(u_i^n,t) ) \psi_i^n,\\
I_2=&\sum^{m}_{k=m_\d + 1} \frac{1}{n^2 \alpha_{kn}} \sum^{n}_{i,j=1}
(\alpha_{kn} \bar{W}_{ij}^{kn} - w_{ij}^{kn}) D_k(v_j^n-v_i^n) \psi_i^n,\\
I_3=&\sum^{m}_{k=1} \frac{1}{n^2 \alpha_{kn}} \sum^{n}_{i,j=1}
w_{ij}^{kn} [D_k(v_j^n-v_i^n)-D_k(u_j^n-u_i^n)]\psi_i^n.
\end{align*}
The Lipschitz continuity of $f$ in $u$ immediately yields
\begin{equation}
|I_1| \leq L_f \| \psi_n \|^2_{2,n}.
\label{eqn:lemb2}
\end{equation}
Using the Lipschitz continuity of $D_k$ and the triangle inequality,
we have
\begin{align}
|I_3| &\leq \sum^{m}_{k=1} \frac{1}{n^2 \alpha_{kn}}
\sum^{n}_{i,j=1}w_{ij}^{kn}L_D(|\psi_{n,i}| + |\psi_{n,j}|)|\psi_{n,i}| \notag\\
&\leq \sum^{m}_{k=1} \frac{L_D}{n^2 \alpha_{kn}}
\biggl( \frac{3}{2}\sum^{n}_{i,j=1}w_{ij}^{kn}\psi_{n,j}^2
+ \frac{1}{2} \sum^{n}_{i,j=1}w_{ij}^{kn}\psi_{n,j}^{2}\biggr)
\label{eqn:lemb3}
\end{align}
since $|\psi_i^n|\,|\psi_j^n| \leq (|\psi_i^n|^2 + |\psi_j^n|^2)/2$.
From \eqref{eqn:dd} we have
\begin{equation}
\begin{split}
&
\frac{1}{n^2 \alpha_{kn}}\sum^{n}_{i,j=1}w_{ij}^{kn}\,\psi^{2}_{n,i}
\le C_1 \| \psi \|^{2}_{n,2},\\
&
\frac{1}{n^2 \alpha_{kn}}\sum^{n}_{i,j=1}w_{ij}^{kn}\,\psi^{2}_{n,j}
\le C_2 \| \psi \|^{2}_{n,2}
\end{split}
\label{eqn:lemb4}
\end{equation}
for $k\in[m_\d]$ since $\alpha_{kn}=1$.
Proceeding as in the proof of Theorem~4.1 of \cite{M19} and using \eqref{eqn:dr},
we obtain
\begin{equation}
\begin{split}
&
\frac{1}{n^2 \alpha_{kn}} \sum^{n}_{i,j=1}w_{ij}^{kn}\psi^{2}_{n,i}
\leq (1 + C_1) \| \psi \|^{2}_{n,2}\quad\text{a.s.},\\
&
\frac{1}{n^2 \alpha_{kn}} \sum^{n}_{i,j=1}w_{ij}^{kn}\psi^{2}_{n,j}
\leq (1 + C_2) \| \psi \|^{2}_{n,2}\quad\text{a.s.}
\end{split}
\label{eqn:lemb5}
\end{equation}
for $k\in[m]\setminus[m_\d]$ and $n\gg 1$.
Substituting \eqref{eqn:lemb4} and \eqref{eqn:lemb5} into \eqref{eqn:lemb3} yields
\begin{equation}
|I_3 | \leq L_D \biggl(2(m-m_\d)
+\frac{3m}{2} C_1+\frac{m}{2}C_2\biggr) \| \psi \|^{2}_{n,2}\quad\text{a.s.}
\label{eqn:lemb6}
\end{equation}

It remains to estimate $|I_2|$.
Define the random variables
\[
Z_{kn,i}(t)= \frac{1}{n} \sum^{n}_{j=1} b_{ij}^{kn} (t) \eta_{ij}^{kn},\quad
i\in[n],
\]
for $k \in [m] \setminus[m_\d]$, where
\[
b_{ij}^{kn}(t)=D_k(v_j^n(t)-v_i^n(t)),\quad
\eta_{ij}^{kn}=w_{ij}^{kn} - \alpha_{kn} \bar{W}_{ij}^{kn}.
\]
Noting that
\[
\frac{Z_{kn,i}(t)}{\alpha_{kn}}\psi_i^n(t)
\leq\frac{1}{2}\biggl(\biggl(\frac{Z_{kn,i}(t)}{\alpha_{kn}}\biggr)^2+\psi_i^n(t)^{2}\biggr),
\]
we have
\begin{align}
| I_2 | = &\biggl| \sum^{m}_{k= m_\d + 1} \frac{1}{n \alpha_{kn}}
\sum^{n}_{i=1} Z_{kni}(t) \psi_i^n(t) \biggr|\notag\\
\leq& \sum^{m}_{k=m_\d + 1} \frac{1}{2 \alpha^{2}_{kn}} \| Z_{kn}(t) \|^{2}_{2,n}
+ \frac{m - m_\d}{2} \| \psi_n(t) \|^{2}_{2,n},
\label{eqn:lemb7}
\end{align}
where $Z_{kn} = (Z_{kn,1},Z_{kn,2}, \ldots, Z_{kn,n})$.
Thus, it follows from \eqref{eqn:lemb1}, \eqref{eqn:lemb2}, \eqref{eqn:lemb6}
 and \eqref{eqn:lemb7} that
\begin{equation}
\frac{d}{dt} \| \psi_n (t) \|^{2}_{2,n} \leq L \| \psi_n(t) \|^{2}_{2,n}
+ \sum^{m}_{k=m_\d + 1} \frac{1}{\alpha_{kn}^2} \| Z_{kn}(t) \|^{2}_{2,n},
\label{eqn:lemb8}
\end{equation}
where
\begin{equation*}
L = 2L_f + mL_D(3C_1+C_2)+(m-m_\d)(4L_D+1).
\end{equation*}

Using Gronwall's inequality for \eqref{eqn:lemb8}, we obtain
\begin{equation}
\max_{t\in[0,T]}\|\psi_n (t)\|_{2,n}^{2}
\leq \sum^{m}_{k=m_\d + 1} \frac{2L}{\alpha^{2}_{kn}} e^{LT} \int^{T}_{0} e^{-Ls}
\| Z_{kn}(s) \|^{2}_{2,n} ds.
\label{eqn:lemb9}
\end{equation}
We proceed as in the proof of Theorem~4.1 of \cite{M19} again
and estimate
\begin{equation}
\int^{T}_{0} e^{-Ls} \| Z_{kn}(s) \|^{2}_{2,n} ds
\leq \frac{C_3}{n}\quad\text{a.s.}
\label{eqn:lemb10}
\end{equation}
for some $C_3>0$,
after some lengthy arguments.
Combining \eqref{eqn:lemb9} and \eqref{eqn:lemb10}
and noting that $\alpha_{kn} = n^{-\gamma_k}$, we have
\begin{align}
\max_{t \in [0,T]} \| \psi_n (t) \|^{2}_{2,n}
\leq&\frac{2LC_3}{n}e^{LT} \sum^{m}_{k = m_\d + 1}\alpha_{kn}^{-2}\notag\\
\leq&2C_3 e^{LT} (m - m_\d) n^{-(1 - 2\gamma_\mathrm{max})}\quad\text{a.s.},
\label{eqn:sec4a}
\end{align}
which yields \eqref{eqn:lemb}.
\end{proof}

We rewrite the averaged model \eqref{eqn:ask} and the initial condition \eqref{eqn:ica} as
\begin{align}
\frac{\partial}{\partial t}v_n(t,x)
=&f(v_n(t,x),t)+\sum_{k=1}^{m_\d}\int_I W_{kn}(x,y)D_k(v_n(t,y)-v_n(t,x))dy\notag\\
& +\sum_{k=m_\d+1}^m\int_I \bar{W}_{kn}(x,y)D_k(v_n(t,y)-v_n(t,x))dy
\label{eqn:ask1}
\end{align}
and
\begin{equation}
v_n(0,x)=g_n(x):=\sum_{i=1}^ng_i^n\mathbf{1}_{I_i^n}(x),
\label{eqn:ica1}
\end{equation}
where
\begin{align*}
&
W_{kn} (x,y)
=\sum_{i=1}^{n}\langle W_k \rangle_{ij}^n\mathbf{1}_{I_i^n\times I_j^n}(x,y),\\
&
\bar{W}_{kn} (x,y)
=\sum_{i=1}^{n}\langle\tilde{W}_k \rangle_{ij}^n\mathbf{1}_{I_i^n\times I_j^n}(x,y).
\end{align*}

\begin{proof}[Proof of Theorem~$\ref{thm:main2}$]
Thanks to Lemma~\ref{lem:b1}, we only have to prove that
the solution $v_n(t,x)$ to the IVP
of the averaged model \eqref{eqn:ask1} with \eqref{eqn:ica1}
converges to the solution $u(t,x)$
of the IVP of the continuum limit \eqref{eqn:csk} with \eqref{eqn:icc}.
We follow the proof of Theorem~5.1 of \cite{M19} with some modifications.

Let $\psi_n(t,x)=u(t,x)-v_n(t,x)$.
Subtracting \eqref{eqn:ask1} from \eqref{eqn:csk},
multiplying the resulting equation by $\psi_n(t,x)$
and integrating it over $I$, we have
\begin{align}
&\frac{1}{2}\frac{d}{dt}\|\psi_n(t,\cdot)\|^2\notag\\
& = \int_I(f(u(t,x),t)-f(v_n(t,x))\psi_n(t,x)dx\notag\\
&\quad
+\sum_{k=1}^m\int_{I^2}W_k(x,y)[D_k(u(t,y)-u(t,x))-D_k(v_n(t,y)-v_n(t,x))]\psi_n(t,x)dxdy\notag\\
&\quad
+\sum_{k=1}^{m_\d}\int_{I^2}(W_k(x,y)-W_{kn}(x,y))D_k(v_n(t,y)-v_n(t,x))]\psi_n(t,x)dxdy\notag\\
&\quad
+\sum_{k=m_\d+1}^m\int_{I^2}(W_k(x,y)-\bar{W}_{kn}(x,y))D_k(v_n(t,y)-v_n(t,x))]\psi_n(t,x)dxdy.
\label{eqn:thm2a}
\end{align}
By the Lipschitz continuity of $f$
\begin{equation}
\biggl|\int_I(f(u(t,x),t)-f(v_n(t,x))\psi_n(t,x)dx\biggr|
\le L_f\|\psi_n(t,\cdot)\|^2.
\label{eqn:thm2b}
\end{equation}
Using Young's inequality and Fubini's theorem
along with the Lipschitz continuity of $D_k$,
\eqref{eqn:assumpx} and \eqref{eqn:assumpy}, we have
\begin{align}
&
\biggl|\int_{I^2}W_k(x,y)[D(u(t,y)-u(t,x))-D(v_n(t,y)-v_n(t,x))]\psi_n(t,x)dxdy\biggr|\notag\\
&
\le L_D\biggl|\int_{I^2}W_k(x,y)(|\psi_n(t,y)|+|\psi_n(t,x)|)|\psi_n(t,x)|dxdy\biggl|\notag\\
&
\le\frac{1}{2}L_D(3C_1+C_2)\|\psi_n(t,\cdot)\|^2
\label{eqn:thm2c}
\end{align}
for $k\in[m]$.
Using Young's inequality and the boundedness of $D_k$, we have
\begin{align}
&
\biggl|\int_{I^2}(W_k(x,y)-W_{kn}(x,y))D(v_n(t,y)-v_n(t,x))\psi_n(t,x)dxdy\biggl|\notag\\
&
\le\frac{C_D}{2}(\|W_n-W_{kn}\|_{L^2(I^2)}^2+\|\psi_n(t,\cdot)\|^2)
\label{eqn:thm2d}
\end{align}
for $k\in[m_\d]$ and
\begin{align}
&
\biggl|\int_{I^2}(W_k(x,y)-\bar{W}_{kn}(x,y))D(v_n(t,y)-v_n(t,x))\psi_n(t,x)dxdy\biggl|\notag\\
&
\le\frac{C_D}{2}(\|W_n-\bar{W}_{kn}\|_{L^2(I^2)}^2+\|\psi_n(t,\cdot)\|^2)
\label{eqn:thm2e}
\end{align}
for $k\in[m]\setminus[m_\d]$.

Combining \eqref{eqn:thm2a}-\eqref{eqn:thm2e}, we have
\begin{align*}
\frac{d}{dt}\|\psi_n(t,\cdot)\|^2
\leq& L_2\|\psi_n(t,\cdot)\|^2
+C_D\sum_{k=1}^{m_\d}\|W_n-W_{kn}\|_{L^2(I^2)}^2\\
& +C_D\sum_{k=m_\d+1}^m\|W_n-\bar{W}_{kn}\|_{L^2(I^2)}^2,
\end{align*}
where $L_2=2L_f+m(L_D(3C_1+C_2)+C_D)$.
Applying Gronwall's inequality to the above equation, we obtain
\begin{align*}
\sup_{t\in[0,T]}\|\psi_n(t,\cdot)\|^2
\le& e^{L_2T}\biggl(\|g-g_n\|+C_D\sum_{k=1}^{m_\d}\|W_n-W_{kn}\|_{L^2(I^2)}^2\\
& +C_D\sum_{k=m_\d+1}^m\|W_n-\bar{W}_{kn}\|_{L^2(I^2)}^2\biggr).
\end{align*}
Obviously, $\|W_n-W_{kn}\|_{L^2(I^2)}\to 0$ as $n\to\infty$ for $k\in[m_\d]$.
We also show that $\|W_n-\bar{W}_{kn}\|_{L^2(I^2)}\to 0$ as $n\to 0$
for $k\in[m]\setminus[m_\d]$,
as in the proof of Theorem 5.1 of \cite{M19}.
This completes the proof.
\end{proof}


\end{document}